\documentclass{article}

\usepackage{amsmath,amssymb,amsfonts,amsthm}
\usepackage{color,graphicx,bbm,verbatim}
\usepackage{mathtools}
\usepackage{epstopdf}
\usepackage{algorithmic}
\usepackage{subfig}
\usepackage{booktabs}
\usepackage{dcolumn}
\usepackage[detect-all]{siunitx}
\usepackage{hyperref}
\usepackage{breakcites}
\usepackage[a4paper,margin=1.2in]{geometry}

\usepackage[capitalize]{cleveref}
\usepackage{enumitem}

\AtBeginDocument{%	
	\crefformat{section}{#2\S#1#3}
	\Crefformat{section}{#2\S#1#3}
	\crefformat{subsection}{#2\S#1#3}
	\Crefformat{subsection}{#2\S#1#3}
	\crefformat{subsubsection}{#2\S#1#3}
	\Crefformat{subsubsection}{#2\S#1#3}
}

\ifpdf
  \DeclareGraphicsExtensions{.eps,.pdf,.png,.jpg}
\else
  \DeclareGraphicsExtensions{.eps}
\fi

\newtheorem{theorem}{Theorem}[section] % use one numbering scheme for all these objects
\newtheorem{lemma}[theorem]{Lemma}
\newtheorem{remark}[theorem]{Remark}
\newtheorem{corollary}[theorem]{Corollary}
\newtheorem{definition}[theorem]{Definition}

\numberwithin{equation}{section}
\numberwithin{table}{section}
\numberwithin{figure}{section}

\def\sep{\mathrm{sep}}

\def\seplam{\sep_\lambda(A,B)}
\def\seplamV{\sep_\lambda^\mathrm{V}(A,B)}
\def\seplamD{\sep_\lambda^\mathrm{D}(A,B)}
\def\fV{f^\mathrm{V}}
\def\fD{f^\mathrm{D}}

\newcommand{\seplamDAB}[2]{\mathrm{sep}_\lambda^\mathrm{D}(#1,#2)}
\newcommand{\seplamVAB}[2]{\mathrm{sep}_\lambda^\mathrm{V}(#1,#2)}

\DeclareMathOperator{\cl}{\mathrm{cl}}
\DeclareMathOperator{\bd}{\mathrm{bd}}
\DeclareMathOperator{\interior}{\mathrm{int}}

\def\matlab{MATLAB}

\def\R{\mathbb{R}}
\def\C{\mathbb{C}}

\def\bigO{\mathcal{O}}
\def\imagunit{\mathbf{i}}
\def\transsym{\mathsf{T}}
\newcommand{\tp}[1][]{^{{#1}\transsym}}
\def\fro{_\mathrm{F}}

\def\hinf{\mathcal{H}_\infty}
\def\linf{\mathcal{L}_\infty}

\DeclareMathOperator{\Arg}{Arg}

\DeclareMathOperator*{\argmin}{arg\,min}

\def\e{\mathrm{e}}
\def\eit{\eix{\theta}}
\def\emit{\emix{\theta}}

\newcommand{\eix}[1]{\e^{\imagunit #1}}
\newcommand{\emix}[1]{\e^{-\imagunit #1}}

\newcommand{\ps}[1][]{%
   \ifthenelse{ \equal{#1}{} }
      {\Lambda_\eps}
      {\Lambda_{#1}}
}

\def\setA{\mathcal{A}}
\def\setB{\mathcal{B}}
\def\setZ{\mathcal{Z}}
\def\setL{\mathcal{L}}
\def\setT{\mathcal{T}}

\def\ray{\mathcal{R}}
\newcommand{\rayth}[1][]{%
   \ifthenelse{ \equal{#1}{} }
      {\ray_\theta}
      {\ray_{#1}}
}

\newcommand{\afn}[1][]{%
   \ifthenelse{ \equal{#1}{} }
      {a_\eps}
      {a_{#1}}
}

\newcommand{\bfn}[1][]{%
   \ifthenelse{ \equal{#1}{} }
      {b_\eps}
      {b_{#1}}
}

\newcommand{\olfn}[1][]{%
   \ifthenelse{ \equal{#1}{} }
      {\ell_\eps}
      {\ell_{#1}}
}

\newcommand{\tfn}[1][]{%
   \ifthenelse{ \equal{#1}{} }
      {t_\eps}
      {t_{#1}}
}

\newcommand{\aroots}[1][]{%
   \ifthenelse{ \equal{#1}{} }
      {\setA_\eps}
      {\setA_{#1}}
}

\newcommand{\broots}[1][]{%
   \ifthenelse{ \equal{#1}{} }
      {\setB_\eps}
      {\setB_{#1}}
}

\newcommand{\abroots}[1][]{%
   \ifthenelse{ \equal{#1}{} }
      {\setZ_\eps}
      {\setZ_{#1}}
}

\newcommand{\olset}[1][]{%
   \ifthenelse{ \equal{#1}{} }
      {\setL_\eps}
      {\setL_{#1}}
}

\newcommand{\troots}[1][]{%
   \ifthenelse{ \equal{#1}{} }
      {\setT_\eps}
      {\setT_{#1}}
}

\def\eps{\varepsilon}
\def\smin{\sigma_{\min}}

\def\n{n}

\def\Hcal{\mathcal{H}}
\def\Hinf{\Hcal_\infty}

\renewcommand{\Re}{\mathrm{Re}\,}
\renewcommand{\Im}{\mathrm{Im}\,}

\def\beq{\begin{equation}}
\def\eeq{\end{equation}}
\def\beqs{\begin{equation*}}
\def\eeqs{\end{equation*}}
\def\bseq{\begin{subequations}}
\def\eseq{\end{subequations}}

\usepackage{caption}
\usepackage{newfloat}
\usepackage[section]{algorithm}
\usepackage{algorithmic}
\DeclareFloatingEnvironment[placement=htb]{algfloat}
\newcommand{\algnote}[1]{\footnotesize \sc{Note: \it#1 } }

\hypersetup{%
    pdftitle={Fast Computation of Sep-Lambda via Interpolation-based Globality Certificates},
    pdfauthor={Tim Mitchell},
    pdfkeywords={sep-lambda, eigenvalue separation, eigenvalue perturbation, pseudospectra}
    }

\title{Fast Computation of {Sep}$_\lambda$ via Interpolation-based\\Globality Certificates}
\author{Tim Mitchell\thanks{
Max Planck Institute for Dynamics of Complex Technical Systems, Magdeburg, 39106 Germany.
Current address: Department of Computer Science, Queens College / CUNY, 65-30
Kissena Boulevard, Flushing, NY 11367, USA,
\href{tim.mitchell@qc.cuny.edu}{\texttt{tim.mitchell@qc.cuny.edu}}.
ORCID: 000-0002-8426-0242.
}}
\date{{\small November 12, 2019\\ Revised: September 13, 2020, May 25, 2021, March 22, 2022, January 21, 2023}}

\begin{document}
\maketitle
\begin{abstract}
Given two square matrices $A$ and $B$, we propose a new approach for computing
the smallest value $\eps \geq 0$ such that $A+E$ and $A+F$ share an eigenvalue,
where \mbox{$\|E\|=\|F\|=\eps$}.
In 2006, Gu and Overton proposed the first algorithm for computing 
this quantity,
called $\seplam$ (``sep-lambda"), using ideas inspired from an earlier algorithm 
of Gu for computing the distance to uncontrollability.
However, the algorithm of Gu and Overton is extremely expensive, 
which limits it to the tiniest of problems, and until now, no other
algorithms have been known. 
Our new algorithm can be orders of magnitude faster and can solve
problems where $A$ and~$B$ are of moderate size.
Moreover, our method consists of many ``embarrassingly parallel" computations, 
and so it can be further accelerated on multi-core hardware.
Finally, we also propose the first algorithm to compute 
an earlier version of sep-lambda where \mbox{$\|E\| + \|F\|=\eps$}.
\end{abstract}

\noindent
\textbf{Keywords:}
sep-lambda, eigenvalue separation, eigenvalue perturbation,
pseudospectra, Hamiltonian matrix

\medskip
\noindent
\textbf{Notation:}
$\| \cdot \|$ denotes the spectral norm, $\smin(\cdot)$ the smallest singular value, 
$\Lambda(\cdot)$ the spectrum, 
$\kappa(\cdot)$ the condition number of a matrix with respect to the spectral norm,
$J = \begin{bsmallmatrix} 0 & I \\ -I & 0 \end{bsmallmatrix}$, a matrix $A \in \C^{2\n\times 2\n}$
is Hamiltonian if $(JA)^* = JA$, 
 $\mu(\cdot)$ the Lebesque measure on $\R$, 
 and $\bd \setA$, $\interior \setA$, and $\cl \setA$ respectively
the boundary, interior, and closure of a set $\setA$.

\section{Introduction}
The quantity $\seplam$ measures how close two square matrices $A \in \C^{m \times m}$ and $B \in \C^{n \times n}$ 
are to sharing a common eigenvalue,
in the sense of how much $A$ and $B$ must be perturbed in order to make this so.
Varah first introduced $\seplam$ in 1979 in \cite{Var79}, and it was subsequently studied by Demmel 
in~\cite{Dem83,Dem86,Dem87c}, although Demmel used a slightly modified version, partly ``because it lets us 
state slightly sharper results later on" \cite[p.~24]{Dem83}.
The two definitions are:
\bseq	
	\label{eq:seplam}
	\begin{alignat}{3}
	\label{eq:seplamv}
	\seplamV  &\coloneqq  \min_{\substack{E \in \C^{m \times m} \\ F \in \C^{n \times n}}} 
		 \{\eps : \Lambda(A + E) \cap  \Lambda(B + F) \neq \varnothing, \, \|E\|  +  \|F\| \leq \eps\}, \\
	\label{eq:seplamd}
	\seplamD &\coloneqq \min_{\substack{E \in \C^{m \times m} \\ F \in \C^{n \times n}}} 
		\{\eps : \Lambda(A  + E) \cap \Lambda(B + F) \neq \varnothing, \max(\|E\|, \, \|F\|) \leq  \eps\}, 
	\end{alignat}
\eseq
with $\seplamV$ denoting Varah's definition and $\seplamD$ denoting Demmel's.
Obviously, they are both zero if $A$ and $B$ share an eigenvalue and both positive otherwise.
When it is not necessary to distinguish between the two variants, 
we drop the superscript and just write $\seplam$.
For convenience, we also assume that $m \leq n$ throughout the paper.

The two $\seplam$ quantities can also be equivalently defined in terms of singular values as well as 
pseudospectra \cite[pp.~348--349]{GuO06a}, where for some $\eps \geq 0$, the \emph{$\eps$-pseudospectrum} of a matrix $A$ is defined
\bseq
\label{eq:pseudo}
\begin{align}
	\ps(A) 
	\coloneqq {}& \{ z \in \C : z \in \Lambda(A+\Delta), \, \|\Delta\| \leq \eps \}, \\
	\label{eq:pseudo_norm_res}
	= {}& \{ z \in \C : \smin(A - zI) \leq \eps \}.
\end{align}
\eseq
The first definition of pseudospectra dates to at least 1967, in Varah's Ph.D. thesis \cite{Var67} with his introduction of an \emph{r-approximate eigenvalue},
while in his 1979 paper on $\seplamV$, Varah used the term \emph{$\eps$-spectrum} for $\ps(A)$.
The current definitive reference on pseudospectra and their applications
is certainly Trefethen and Embree's well-known book on the topic \cite{TreE05}.  
The term ``pseudospectrum" was actually coined by Trefethen in 1990 \cite[Ch.~6]{TreE05}, 23 years
after Varah's thesis, although it now considered the standard name. 

The singular-value-based definitions of $\seplam$ are
\bseq	
	\label{eq:min_z}
	\begin{alignat}{3}
	\label{eq:min_zv}
	\seplamV &= \min_{z \in \C} \{ \smin(A - zI) + \smin(B - zI) \} &&\eqqcolon \min_{z \in \C} \fV(z), \\
	\label{eq:min_zd}
	\seplamD &= \min_{z \in \C} \max \{ \smin(A - zI), \, \smin(B - zI) \} &&\eqqcolon \min_{z \in \C} \fD(z).
	\end{alignat}
\eseq
For equivalent pseudospectral-based definitions of $\seplam$, we have 
\bseq	
	\label{eq:min_e}
	\begin{alignat}{3}
	\label{eq:min_ev}
	\seplamV &= \min_{\eps_1,\eps_2 \geq 0} && \{\eps_1 + \eps_2 : \ps[\eps_1](A) \cap \ps[\eps_2](B) \neq \varnothing \}, \\
	\label{eq:min_ed}
	\seplamD &= \ \ \min_{\eps \geq 0} && \{\eps : \ps(A) \cap \ps(B) \neq \varnothing \}.
	\end{alignat}
\eseq
If $\eps \geq \seplamD$ holds, 
then $\interior \ps(A) \cap \interior \ps(B) = \varnothing$ is a sufficient condition for \mbox{$\eps=\seplamD$}.
In contrast, while $\interior \ps[\eps_1](A) \cap \interior \ps[\eps_2](B) = \varnothing$ is a necessary condition for $\eps_1+\eps_2=\seplamV$ to hold, it is not a sufficient condition.
This is because one can continuously adjust $\eps_1$~and~$\eps_2$
such that the two pseudospectra always touch but never have interior points in common.
For example, suppose that \mbox{$\seplamV > 0$}, and let $\hat \eps_1$ be such that $\ps[\hat \eps_1](A)$ and $\ps[0](B)$ only touch,
i.e., an eigenvalue of $B$ is in $\bd \ps[\hat \eps_1](A)$ but 
\mbox{$\interior \ps[\hat \eps_1](A) \cap \ps[0](B) = \varnothing$}.
In the same fashion, let $\hat \eps_2$ be such that~$\ps[0](A)$ and~$\ps[\hat \eps_2](B)$ only touch.
Then by continuity of pseudospectra,
it is clear that the 2D point $(\eps_1,\eps_2)$ can be continuously adjusted between 
point $(\hat \eps_1,0)$ and point $(0,\hat \eps_2)$ 
such that $\bd \ps[\eps_1](A) \cap \bd \ps[\eps_2](B) \neq \varnothing$ and
\mbox{$\interior \ps[\eps_1](A) \cap \interior \ps[\eps_2](B) = \varnothing$}
both always hold.

Varah called $\seplamV$ the \emph{spectrum separation} in \cite[Definition~3.2]{Var79}
due to its pseudospectral underpinnings; in fact, in his definition, he used
the form given in~\eqref{eq:min_ev}, not the other two alternatives.
His motivation in defining $\seplamV$ was its connection to the 
sensitivity of solving the Sylvester equation:
\beq
	\label{eq:sylv}
	AX - XB = C,
\eeq
where $X,C \in \C^{m \times n}$ and
\eqref{eq:sylv} has a unique solution if and only if $A$ and $B$ have no common eigenvalue.
As Varah noted~\cite[p.~216]{Var79}, the sensitivity of a solution to \eqref{eq:sylv} is inversely proportional
to the \emph{separation of~$A$ and~$B$}:
\[
	\sep(A,B) \coloneqq \min_{\|X\|\fro = 1} \| AX - XB \|\fro = \smin(I_n \otimes A - B\tp \otimes I_m),
\]
a quantity which Stewart had earlier introduced for studying invariant subspaces~\cite[Definition~4.5]{Ste73b}.
It holds that $0 \leq \sep(A,B) \leq \min_{\lambda \in \Lambda(A), \mu \in \Lambda(B)} | \lambda - \mu |$,
and clearly, the lower bound is attained if and only if  $A$ and $B$ have an eigenvalue in common,
while the upper bound is attained if $A$ and $B$ are both normal.
However, Varah stressed that if $A$ or $B$ is nonnormal, 
then $\sep(A,B)$ can be very close to zero, e.g., machine precision, 
even if the eigenvalues of $A$ and $B$ are well separated,
and that $\sep(A,B)$ is often orders of magnitude smaller than $\seplamV$.
In~1993, Higham's thorough error analysis for solving~\eqref{eq:sylv} numerically
showed that bounding the error of a computed solution in terms of $\sep(A,B)^{-1}$
can sometimes ``severely overestimate the effect of a perturbation on the data when only $A$ and $B$ are perturbed, 
because it does not take account of the special structure of the problem" \cite[p.~133]{Hig93},
while simultaneously presenting an alternative error bound that remedies this deficiency.  
A few years later, Simoncini used $\seplamV$ and pseudospectra in her analysis of
solving~\eqref{eq:sylv} via a Galerkin method~\cite{Sim96a}.

Meanwhile, Demmel initial interest in (his version of) $\seplam$ was for problem of computing stable 
eigendecompositions~\cite{Dem83,Dem86}, but in an entirely different context~\cite{Dem87c}, 
he subsequently used $\seplamD$ to disprove two conjectures respectively made by himself and Van Loan related to the (then unsolved)
problem of computing the distance to instability of a stable matrix.
Following in the spirit of using $\seplam$ in the analysis of the 
stability of invariant subspaces of matrices~\cite{Var79,Dem83,Dem86},
 Karow and Kressner used $\seplamD$ in 2014 as a tool in deriving improved perturbation bounds~\cite{KarK14}.
Most recently in 2021, Roy~et~al.~\cite{RoyKBetal20} used~$\seplamV$ in 
connection with approximating pseudospectra of block triangular matrices;
in this case, the value of $\seplamV$ can be used to construct several different outer approximations to pseudospectra
of these structured matrices.

In terms of computing $\seplam$, 
 to the best of our knowledge, only a single algorithm has been given so far for $\seplamD$, due to Gu and Overton in 2006~\cite{GuO06a},
while no algorithms have appeared to date for $\seplamV$.
Nevertheless, computing $\seplamD$ can at least approximate
$\seplamV$ to within a factor of two since
\beq
	\label{eq:varah_bound}
	\frac{1}{2}\seplamV \leq \seplamD \leq \seplamV,
\eeq
which is simply a special case of the relation $\tfrac{1}{n}\|x\|_1 \leq \|x\|_\infty \leq \|x\|_1$ for $x \in \C^n$
obtained by respectively identifying $\seplamV$ and $\seplamD$ with the $1$-norm and $\infty$-norm 
of \mbox{$(\|E\|, \|F\|)\tp \in \R^2$}.

It is easy to obtain upper bounds for $\seplam$ by simply evaluating $\fV$ and/or $\fD$ defined in \eqref{eq:min_z}
at any points $z \in \C$, or better, by applying (nonsmooth) optimization techniques to find local minimizers of them.
Due to the $\max$ function in $\fD$, it is typically nonsmooth at minimizers, 
while~$\fV$ will be nonsmooth at a minimizer if that minimizer happens to coincide 
with an eigenvalue of~$A$ or $B$, which as Gu and Overton mentioned, is often the case for $\seplamV$.
Despite the potential nonsmoothness, $\fV$ and $\fD$ are rather straightforward functions 
in just two real variables (via $z = x + \imagunit y$),
whose function values and gradients (assuming $z$ is a point where they are differentiable)
can be obtained via computing~$\smin(A - zI)$ and $\smin(B - zI)$
and their corresponding left and right singular vectors.
When $A$ and $B$ are large and sparse, it is often still possible to efficiently compute 
$\fV$ and $\fD$ and their gradients via sparse methods.
Nevertheless, finding local minimizers of \eqref{eq:min_z} provides no guarantees for computing~$\seplam$,
particularly since these problems may have many different local minima and 
the locally optimal function values associated with these minima may be very different.
Moreover, in applications that use distances measures such as $\seplam$, 
obtaining an upper bound via local optimization is generally much less useful 
than either computing the actual measure or a lower bound to it.    
Indeed, in motivating their algorithm for $\seplamD$, Gu and Overton aptly remarked~\cite[p.~350]{GuO06a}: 
``the inability to verify global optimality [of minimizers of $\fD$] remains a stumbling block 
preventing the computation of $\seplam$, or even the assessment of the quality of upper bounds, 
via optimization" and ``in applications, lower bounds for such distance functions are
more important than upper bounds, as they provide `safety margins.'"

In this paper, we propose a new and much faster method to compute $\seplamD$ to arbitrary accuracy,
using properties of pseudospectra, local optimization techniques, 
and a new methodology that we recently introduced in \cite{Mit21} 
for finding global optimizers of singular value functions in two real variables.
This new approach, called \emph{interpolation-based globality certificates},
can be orders of magnitude faster than existing techniques and also avoids numerical difficulties inherent
in older approaches.
A modified version of our new $\seplamD$ algorithm also produces estimates of $\seplamV$ 
with stronger guarantees than those obtained by optimization; specifically,
this modified method produces locally optimal upper bounds~\mbox{$\tilde \eps = \eps_1 + \eps_2 \geq \seplamV$}
such that $\interior \ps[\eps_1](A) \cap \interior \ps[\eps_2](B) = \varnothing$, 
which is a necessary condition for $\tilde \eps = \seplamV$ to hold, but which optimization alone does not guarantee.
Finally, we also propose a separate algorithm that is the first to compute $\seplamV$.

The paper is organized as follows. 
In~\cref{sec:background}, we give a brief overview 
of Gu and Overton's method for~$\seplamD$ \cite{GuO06a} and explain its shortcomings.
Then, in \cref{sec:algorithm}, we give a high-level description of our new \emph{optimization-with-restarts} method
and an introduction to the ideas underlying interpolation-based globality certificates.  
As our new globality certificate for $\seplamD$ is quite different and significantly more complicated 
than those we devised for computing Kreiss constants and the distance to uncontrollability in~\cite{Mit21}, 
we develop the necessary theoretical statements and components over three separate stages 
in~\cref{sec:components}, \cref{sec:overlap}, and \cref{sec:support}.
In \cref{sec:implement}, we describe how to implement our completed algorithm and give its overall work complexity.
We then turn to Varah's sep-lambda in~\cref{sec:varah}.
Numerical experiments are presented in \cref{sec:experiments}, with concluding remarks given in \cref{sec:conclusion}.

\section{Gu and Overton's method to compute $\seplamD$ and its limitations}
\label{sec:background}
The algorithm of Gu and Overton for computing $\seplamD$ is particularly expensive: 
it is \mbox{$\bigO((m+n)m^3n^3)$} work, e.g.,~$\bigO(n^7)$ when $m=n$,
which makes it intractable for all but the tiniest of problems.
The core of their method is a pair of related tests, each of which is 
inspired by a novel but expensive 2D level-set test developed earlier by 
Gu for estimating the distance to uncontrollability \cite{Gu00}.
The cost of each test is dominated by solving 
an associated generalized eigenvalue problem of order $4mn$,
which is $\bigO(m^3n^3)$ work 
when using standard dense eigensolvers.\footnote{With respect to the usual convention of treating the computation of eigenvalues as an atomic operation with cubic work complexity, which we use throughout this paper.}
Given some $\eps \geq 0$, the first test (\cite[Algorithm~1]{GuO06a}) checks whether 
the $\eps$-level sets of~$\smin(A - zI)$ and~$\smin(B - zI)$ have any points in common. 
If this is indeed the case, then clearly~$\eps \geq \seplamD$ must hold.
However, if there are no level-set points in common, one cannot conclude that $\eps < \seplamD$ holds.
For example, having no shared level-set points may just be a consequence of $\ps(A)$ being a subset of $\interior \ps(B)$ or vice versa,
in which case, clearly~$\eps > \seplamD$ holds.
To get around this difficulty, Gu and Overton devised an initialization procedure~(\cite[Algorithm~2]{GuO06a}), 
which invokes their second test many times in order to compute 
an upper bound~$\eps_\mathrm{ub}$ such that for all~$\eps < \eps_\mathrm{ub}$,
no connected component of $\ps(A)$ can be strictly inside a component of $\ps(B)$ or vice versa.\footnote{In~\cite{GuO06a},
Gu and Overton state that this ``not strictly inside" property holds for $\eps \leq \eps_\mathrm{ub}$,
but actually this inequality should be strict.
Near the top of~\cite[p.~354]{GuO06a}, it is claimed that ``$\eps = \smin(A - zI) > \smin(B - zI)$" holds,
where $z \in \bd \ps(A)$ and~$z \in \interior \ps(B)$.
However, per~\cite[p.~31]{ArmGV17}, there can exist a finite number of points $z \in \interior \ps(B)$
such that $\smin(B - zI) = \eps$, and so the ``not strictly inside" claim may or may not hold when~$\eps = \eps_\mathrm{ub}$.
Fortunately, with inexact arithmetic, there is essentially no practical consequence of this small oversight, 
while the theory in~\cite{GuO06a} is corrected merely by replacing $\leq$ with $<$.
}
With this possibility excluded, i.e., $\eps < \eps_\mathrm{ub}$, 
the outcome of the first test then does indicates whether or not $\eps < \seplamD$ holds.
Gu and Overton's overall method \cite[Algorithm~3]{GuO06a} thus first computes $\eps_\mathrm{ub}$ via their initialization procedure
and then uses their first test to power a bisection iteration that converges to $\seplamD$.
The entire bisection phase of their algorithm remains $\bigO(m^3n^3)$ work, since the number of bisection steps can
be taken as a constant, but the initialization phase to compute the necessary $\eps_\mathrm{ub}$ 
involves invoking the second test for $(m+n)$ different parameter values, i.e., 
it solves~$(m+n)$ different generalized eigenvalue problems of order $4mn$.
Hence, the cost of their entire method is dominated by the initialization procedure,
and the total asymptotic work complexity is~$\bigO((m+n)m^3n^3)$.

In their concluding remarks \cite[p.~358]{GuO06a},
Gu and Overton noted that the faster divide-and-conquer technique of~\cite{GuMOetal06}
for computing the distance to controllability 
could potentially be adapted to $\seplamD$, writing that ``Although there are some
inevitable difficulties with the numerical stability of this approach, the complexity
drops significantly."  
Indeed, when $m=n$, adapting this divide-and-conquer approach would 
bring down the $\bigO(n^7)$ work complexity of their algorithm for $\seplamD$
to $\bigO(n^5)$ on average and~$\bigO(n^6)$ in the worst case.
However, this has not been implemented, and in our own experience of adapting this
divide-and-conquer technique to other algorithms, we have observed that 
doing so can indeed come at the cost of significantly worse reliability due to numerical issues; 
see \cite[section~8]{Mit20a}.

Even with dense eigensolvers, Gu and Overton's method can be susceptible to numerical difficulties.
A primary concern is that the first test (used for bisection)
actually requires being able to assert whether or not two matrices have an eigenvalue in common.
If eigenvalues can be computed exactly (which is possible in some cases, e.g., a diagonal matrix),
then testing whether two matrices share an eigenvalue can be done without issues.
However, in a practical code,
computed eigenvalues will have rounding errors, and so one must generally resort 
to using a tolerance in order to carry out this test.
But this also means that it is possible for the test to incorrectly assert that two eigenvalues are the same
when they should only be considered close or vice versa.
This is critical because 
the \emph{binary decision} of bisection hinges upon the outcome of this numerical test.
Making the wrong choice about the eigenvalues can cause bisection to 
erroneously update a lower or upper bound,
which in turn can result in a significant or even complete loss of accuracy in the computed estimate.
The distance-to-uncontrollability methods of~\cite{Gu00,BurLO04,GuMOetal06} also 
have the same numerical pitfall.  In the context of computing Kreiss constants via 2D level-set tests \cite{Mit20a},
we recently proposed an improved procedure that does not require checking for shared eigenvalues, 
and as such, it is much more reliable in practice; see~\cite[Key~Remark~6.3]{Mit20a}.
Our improved technique can also be used to improve the reliability of the aforementioned distance-to-uncontrollability algorithms,
but it does not appear to be applicable for Gu and Overton's algorithm for~$\seplamD$.
The fundamental difference in the $\seplamD$ setting is that Gu and Overton's first test is based upon checking whether
or not the $\eps$-level sets of two \emph{different} functions,~$\smin(A - zI)$ and~$\smin(B - zI)$, have any points in common,
whereas for the other quantities, pairs of points on a given level set of a \emph{single} function are sought.

Finally, another way to provide some speedup to Gu and Overton's method would be to replace the bisection 
phase with an \emph{optimization-with-restarts} iteration.  In this case, a minimizer $\tilde z$ of $\fD$ with~$\eps = \fD(\tilde z)$ would be found using some nonsmooth optimization solver,
and then, assuming~\mbox{$\eps < \eps_\mathrm{ub}$}, Gu and Overton's first test (\cite[Algorithm~1]{GuO06a})
would be used to assert whether or not~$\tilde z$ is a global minimizer of $\fD$. 
If so, then $\seplamD = \eps$ and the computation is done.
Otherwise, recalling that the first test computes the points $z \in \C$ such that~\mbox{$\smin(A - zI) = \smin(B - zI) = \eps$},
local optimization can be restarted from these points in order to find a better (lower) minimizer.
Any such optimization-with-restarts method must monotonically converge to $\seplamD$
within a finite number of restarts because $\fD$ only has a finite number of locally minimal function
values, due to $\fD$ being semialgebraic. 
However, there are some issues with this modification.  
The main limitation is that it neither accelerates nor removes the need for the initialization procedure 
for obtaining $\eps_\mathrm{ub}$, which is $\bigO(n^7)$ work, 
while the subsequent convergent phase of either bisection or optimization-with-restarts is $\bigO(n^6)$ work.  
Consequently, any speedups will be both quite small and limited to the smallest values of $n$, while being essentially nonexistent
for larger $n$.  
Another problem is that theory for nonsmooth optimization typically requires that solvers are initialized at points 
where the function is differentiable (see, e.g., \cite{BurCLetal20,LewO13,CurMO17}), 
but by its nature, the points computed by Gu and Overton's first test are all points
where $\fD$ will almost certainly be nonsmooth.
Hence, there may be issues in restarting optimization via these points, and depending on the exact solver and problem,
we have observed that solvers can indeed stagnate at these initial points.  
It is not entirely clear how to best overcome this latter issue, but for us, it is not a priority.
Instead, the focus of this paper is to propose an entirely different approach
to computing $\seplamD$ that allows us to use optimization-with-restarts without any
of the aforementioned drawbacks of extremely high costs, expensive initialization procedures,
and various numerical and technical issues.

\section{A high-level overview of our new $\seplamD$ algorithm}
\label{sec:algorithm}
To find a global minimizer of $\fD$, a global optimization problem in two real variables,
we will instead develop our optimization-with-restarts algorithm using 
\emph{interpolation-based globality certificates}~\cite{Mit21}.
The core task in developing such a method  
is to devise a generally continuous function (i.e., it may have some jumps) in one real variable that, given an estimate \emph{greater} than the globally minimal value, has an identifiable subset of its domain with \emph{positive measure} that 
provides a guaranteed way of locating new starting points for another round of optimization.
When an estimate is globally minimal, this function should alternatively assert this fact somehow,
e.g., by determining that the aforementioned subset is either empty or has measure zero.
By sufficiently well approximating this function globally via a piecewise polynomial interpolant (this interpolant may also have jumps),
e.g., by using Chebfun\footnote{Available at \url{https://www.chebfun.org}.}~\cite{DriHT14},
it is then possible to quickly check for the existence 
of the aforementioned positive measure subset, whose presence indicates that the estimate
is not globally optimal.  In fact, Chebfun can efficiently compute the precise set of intervals corresponding to this subset.
When the estimate is too large, the property that there exists a subset of positive measure associated 
with new starting points is crucial for two reasons.
First, it means that encountering this subset during the interpolation process is not  a probability zero event,
and so if the function is well approximated, this subset will be detected. 
Second, optimization can be immediately restarted once any points in this subset are discovered, 
and so high-fidelity interpolants will often not be needed.
As a result, restarts tend to be very inexpensive, while high-fidelity approximation is generally only needed for the final interpolant, 
which asserts that global convergence has indeed been obtained.
Moreover, in practice only a handful of restarts are typically needed.
Besides overall efficiency, interpolation-based globality certificates are inherently amenable to additional acceleration 
via parallel processing (see \cite[section~5.2]{Mit21}), while also being quite numerically robust compared to other techniques.
There are several reasons for this latter property, but one is that by the nature of interpolation, global 
convergence is assessed as the result over many computations, whereas other approaches often rely upon  
a single computation that may result in an erroneous conclusion due to rounding errors;
for more details, see~\cite[sections~1.3 and 2.3]{Mit21}.

Given some estimate $\eps \geq \seplamD$, 
in the next sections, we consider the problem of what function~\mbox{$d_\eps : \R \to \R$} to devise for our globality certificate
for either finding new points for restarting optimization or asserting whether $\eps = \seplamD$ holds.
The function~$d_\eps$ should be reasonably well behaved and relatively cheap to evaluate,
as otherwise approximating it could be prohibitively expensive and/or difficult.
But as mentioned above, Chebfun can efficiently handle nonsmooth and discontinuous functions;
\cite[pp.~905--906]{PacPT09} describes the algorithm that Chebfun uses to efficiently detect 
discontinuities, either jumps in the function values or derivatives, which
allows Chebfun to work around these difficult points during its approximation process.
Consequently, we do not have to limit ourselves to smooth continuous candidates for $d_\eps$.
The function that we will propose is based on detecting 
whether or not $\interior \ps(A) \cap \interior \ps(B)$ is empty
and asserts that $\eps > \seplamD$ if and only if 
$\min_{\theta \in (-\pi,\pi]} d_\eps(\theta) < 0$ holds.
Moreover, our certificate for detecting whether 
\mbox{$\interior \ps(A) \cap \interior \ps(B) = \varnothing$} holds
works for any value $\eps > \seplamD$.
As we will explain, our~$d_\eps$-based globality certificate 
incurs $\bigO(kn^3)$ work (recall that we assume $m \leq n$),
where $k$ is the number of function evaluations required to sufficiently approximate $d_\eps$.
Furthermore, our certificate also becomes more efficient the larger $\eps$ is,
i.e., relatively few function evaluations are needed to approximate $d_\eps$ when $\eps \gg \seplamD$
as compared to when $\eps \approx \seplamD$.

\begin{remark}
\label{rem:any_overlap}
Although the first test of Gu and Overton (\cite[Algorithm~1]{GuO06a})
also detects if \mbox{$\interior \ps(A) \cap \interior \ps(B) = \varnothing$} holds,
note that their test is both more limited in scope and more expensive than our~$d_\eps$-based certificate.
Gu and Overton's first test (a) requires that $\eps < \eps_\mathrm{ub}$ holds in order to use it, with $\eps_\mathrm{ub}$
being very expensive to obtain, and (b) 
does the same amount of work regardless of the value of~$\eps$;
again, when~$m=n$, 
 computing $\eps_\mathrm{ub}$ is $\bigO(n^7)$ work, while \cite[Algorithm~1]{GuO06a} is $\bigO(n^6)$ work.
\end{remark}

\section{Locating pseudospectral components}
\label{sec:components}
We now work on defining $d_\eps$ and establishing its properties, which is done over three sections.
This section follows similarly to \cite[sections~2--4]{Mit21}, 
where we first proposed interpolation-based globality certificates to find level-set components
as tools for computing Kreiss constants and the distance to uncontrollability.
Here we adapt these ideas to locating pseudospectral components, and throughout this section, 
we provide specific references to counterparts in \cite{Mit21}.  
However, as will be seen, 
computing~$\seplamD$ is more complicated than computing these other quantities, and 
so the additional tools that we develop in \cref{sec:overlap} and \cref{sec:support} will also be needed.

Given a matrix $A \in \C^{m \times m}$, $\eps \geq 0$, and some $z_0 \in \C$
such that $\eps$ is not a singular value of~$A - z_0 I$,
in this section we propose a way of 
determining which rays emanating from $z_0$ intersect with~$\ps(A)$ and which do not.
We define the ray emanating from $z_0$ specified by angle $\theta \in \R$ as
\beq
	\label{eq:ray}
	\rayth \coloneqq \{ z_0 + r\eit \in \C : r \in \R, \, r > 0\}.
\eeq
As we will explain momentarily, our assumption on $\eps$ ensures that a  
condition needed by our method indeed holds; relatedly, 
our assumption also ensures that the ``search point" $z_0$ is not on the boundary of~$\ps(A)$.
Consider the following function parameterized in polar coordinates:
\beq
	\label{eq:fA}
	f_A(r,\theta) = \smin(A-(z_0 + r\eit) I)) 
	= \smin ( F_A(r,\theta)),
	\ \ \text{where} \ \
 	F_A(r,\theta) = \imagunit \emit (A - z_0 I) - \imagunit r I
\eeq
and the second equality above holds since multiplication by a unitary scalar does not alter 
the singular values of a matrix.
Note that~$\bd \ps(A)$ is contained in the $\eps$-level set of $f_A$.
The following pair of results give us a way to determine whether or not $\rayth$ and  $\ps(A)$ intersect,
and when they do, to also calculate all the points in $\rayth \cap \bd \ps(A)$.
The first of these two results is yet another variation of the 1D level-set technique 
Byers introduced in order to develop the first method for computing the distance to instability in~1988~\cite{Bye88},
a powerful tool which we and many others have adapted, extended, or used to develop
1D and 2D methods to compute various quantities.
Applications include the
$\hinf$ and $\linf$ norms~\cite{BoyBK89,BoyB90,BruS90,BenSV12,BenM18a},
distance to uncontrollability~\cite{Bye90a,GaoN93,Gu00,Mit21},
numerical radius~\cite{HeW97,MenO05,Mit20},
pseudospectral (or spectral value set) abscissa and radius~\cite{BurLO03,MenO05,BenM19},
Kreiss constants~\cite{Mit20a,Mit21},
as well as the optimization of passive systems~\cite{MehV20,MehV20a,MitV21}.

% Due to the [], need to define this separately to pass to \begin{lemma}[string]
\def\citeThmsOne{cf. \cite[Theorems~2.1, 3.1, and 4.1]{Mit21}}
\begin{lemma}[\citeThmsOne]
Let $A \in \C^{m \times m}$, $\eps \geq 0$, $z_0 \in \C$, and $r, \theta \in \R$.
Then $\eps$ is a singular value of $F_A(r,\theta)$ defined in \eqref{eq:fA} if and only if $\imagunit r$ is an eigenvalue of
the Hamiltonian matrix 
\beq
	\label{eq:eigCD}
	C_\theta \coloneqq
	\begin{bmatrix}
		\imagunit \emit(A - z_0 I)  	& -\eps I  \\
		\eps I 				& \imagunit \eit (A - z_0 I )^* \\
	\end{bmatrix}.
\eeq
\label[lemma]{lem:eig_ham}
\end{lemma}
\begin{proof}
Suppose that $\eps$ is a singular value of $F_A(r,\theta)$ with left and right singular vectors $u$ and $v$.  Then
\[
	\eps 
	\begin{bmatrix}
		u \\ v
	\end{bmatrix}
	= 
	\begin{bmatrix}
		F_A(r,\theta)  	& 0 \\
		0 			& F_A(r,\theta)^*
	\end{bmatrix}
	\begin{bmatrix}
		v \\ u
	\end{bmatrix}
	= 
	\begin{bmatrix}
		\imagunit \emit (A - z_0 I) 	& 0	\\
		0 					& -\imagunit \eit (A - z_0 I)^* 
	\end{bmatrix}
	\begin{bmatrix}
		v \\ u
	\end{bmatrix}
	+ 	
	\imagunit r
	\begin{bmatrix}
		- I   	& 0 \\
		0 		& I 
	\end{bmatrix}
	\begin{bmatrix}
		v \\ u
	\end{bmatrix}.
\]
Rearranging terms, using the fact that
$\begin{bsmallmatrix} u \\ v \end{bsmallmatrix} 
= 
\begin{bsmallmatrix} 0 & I \\ I & 0 \end{bsmallmatrix} 
\begin{bsmallmatrix} v \\ u \end{bsmallmatrix} 
$, and multiplying the bottom block row by~$-1$, 
we obtain $C_\theta \begin{bsmallmatrix} v \\ u \end{bsmallmatrix}  = \imagunit r \begin{bsmallmatrix} v \\ u \end{bsmallmatrix}$.
\hfill
\end{proof}

\begin{corollary}
Let $A \in \C^{m \times m}$, $\eps \geq 0$, $z_0 \in \C$, $r,\theta \in \R$, 
and $\rayth$ be the ray defined by \eqref{eq:ray}.
Then $\rayth \cap \ps(A) \neq \varnothing$
if and only if $\imagunit r$ is an eigenvalue of $C_\theta$ with $r > 0$.
\label[corollary]{cor:ray_thru}
\end{corollary}
\begin{proof}
Suppose that $\rayth$ and 
$\ps(A)$ intersect.  As $\ps(A)$ is bounded, there exists an $r > 0$
such that the point $z_0 + r \eit$ is also on the boundary of $\ps(A)$, and so
$\smin (F_A(r,\theta)) = \eps$.  Thus by~\cref{lem:eig_ham},
$\imagunit r$ is an eigenvalue of $C_\theta$.
Now suppose $C_\theta$ has some eigenvalue $\imagunit r$ with $r > 0$.
Again by \cref{lem:eig_ham}, $\eps$ must then be a singular value of $F_A(r, \theta)$
but not necessarily the smallest one.
Thus, $\smin( F_A(r, \theta) ) = \hat \eps \leq \eps$ and so it follows that $z_0 + r \eit$ is 
in $\ps[\hat \eps](A) \subseteq \ps(A)$.
\hfill
\end{proof}

For any $z_0 + r \eit \in \bd \ps(A)$ with $r > 0$,
clearly $\smin(F_A(r,\theta)) = \eps$, and so by \cref{lem:eig_ham}, $\imagunit r \in \Lambda(C_\theta)$.  
Hence, via computing all of the imaginary eigenvalues of $C_\theta$, 
\Cref{lem:eig_ham} provides a way to calculate all of the points in~$\rayth \cap \bd \ps(A)$.
However, note that if~$\imagunit r \in \Lambda(C_\theta)$ with $r > 0$, 
then $z_0 + r \eit$ may or may not be on $\bd \ps(A)$.
There are two reasons for this.  First, per the proof of \cref{cor:ray_thru},
$\eps$ may not be the smallest singular value of $F_A(r,\theta)$, 
in which case~$z_0 + r \eit \in \ps[\hat \eps](A)$ for some $\hat \eps < \eps$.
Second, there can exist a \emph{finite number} of points $z \in \ps(A)$ 
such that~$z \not\in \bd \ps(A)$ but~$\smin(F_A(r,\theta)) = \eps$ nevertheless holds; see \cite[p.~31]{ArmGV17}.

\cref{cor:ray_thru} can be stated more strongly, i.e., in terms of a line intersecting 
$\ps(A)$, since $\imagunit r_\mathrm{neg}$ with~$r_\mathrm{neg} < 0$
is an eigenvalue of $C_\theta$ if and only if $\imagunit |r_\mathrm{neg}|$ 
is an eigenvalue of $C_{\theta + \pi}$.
However, for developing the theoretical concepts for our algorithm, 
it will be more intuitive and simpler to work with the notion of rays emanating from $z_0$ for the time being.
For a code, it does make sense to take advantage of all the imaginary eigenvalues of $C_\theta$,
and we describe how this is done, along with other implementation details, in \cref{sec:implement}.
Regarding the spectrum of $C_\theta$, also note that since $C_\theta$ is Hamiltonian, 
its eigenvalues are symmetric with respect to the imaginary axis.  
Eigenvalues of \emph{real} Hamiltonian matrices have additional symmetry with respect to the real axis,
but this is generally not the case for the spectrum of $C_\theta$ due to 
$C_\theta$ being generically complex valued (even if $A$ is real).
Structure-preserving eigensolvers exist, e.g., \cite{BenMX98b}, that preserve this eigenvalue symmetry 
numerically.

\begin{remark}
While~\cref{lem:eig_ham} pertains to the eigenvalues of a single Hamiltonian matrix,
the analogous~\cite[Theorems~2.1, 3.1, and 4.1]{Mit21} used for computing Kreiss constants
and the distance to uncontrollability via interpolation-based globality certificates are in terms
of the eigenvalues of certain \emph{structured matrix pencils}.
For the case of Kreiss constants~\cite[Theorems~2.1 and 3.1]{Mit21}, the associated matrix pencils 
include parametric matrices that can be singular, and 
so these generalized eigenvalue problems cannot be reduced to standard eigenvalue problems.
However, the matrix pencil for the distance to uncontrollability does permit such a reduction, 
i.e.,~\cite[Theorem~4.1]{Mit21} can be simplified to be in terms of the eigenvalues of the complex Hamiltonian
matrix 
\beq
	\label{eq:dtu_ham}
	\begin{bmatrix} \imagunit \emit A & \widetilde B \\ \gamma I & \imagunit \eit A^* \end{bmatrix},
\eeq
where in~\eqref{eq:dtu_ham}, matrices $A$ and $\widetilde B$ and scalars $\gamma$ and $\theta$ are defined in~\cite[Theorem~4.1]{Mit21}.
\end{remark}

As we will soon see, we will need to preclude the possibility of zero being an eigenvalue of $C_\theta$.
The following straightforward result shows that our assumption on $\eps$ not being a singular value of~$A - z_0I$
accomplishes this.

% Due to the [], need to define this separately to pass to \begin{lemma}[string]
\def\citeThmsTwo{cf. \cite[Theorems~2.4, 3.3, and 4.4]{Mit21}}
\begin{lemma}[\citeThmsTwo]
Let $A \in \C^{m \times m}$, $\eps \in \R$, $z_0 \in \C$, and $\theta \in \R$.  Then 
the matrix $C_\theta$ defined in~\eqref{eq:eigCD} has zero as an eigenvalue 
if and only if the matrix $(A - z_0 I)(A - z_0 I)^*$ has $\eps^2$ as an eigenvalue.
\label[lemma]{lem:zero_eig}
\end{lemma}
\begin{proof}
Since the blocks of $C_\theta$ are all square matrices of the same size, and 
the lower two blocks, $\eps I$ and $\imagunit \eit(A-z_0I)^*$, commute, we have that 
\[
	\det(C_\theta) = \det(-(A-z_0I)(A-z_0I)^* - (-\eps I)(\eps I)) = \det( (A-z_0I)(A-z_0I)^* - \eps^2 I),
\]
thus proving the if-and-only-if equivalence.
\hfill
\end{proof}

We are now ready to present the first major component in our construction of $d_\eps$.
Given $\eps \geq 0$ specifying the $\eps$-pseudospectrum of $A$, and $z_0 \in \C$ such 
that $\eps$ is not a singular value of $A - z_0 I$,
we define the function $\afn : (-\pi,\pi] \to [0,\pi^2]$ and 
associated set (cf. \cite[Equations~(2.4), (3.4), and (4.4)]{Mit21}):
\bseq
	\label{eq:psa_comps}
	\begin{align}
	\label{eq:psa_fn}
	\afn(\theta) &\coloneqq
	\min \{ \Arg(-\imagunit \lambda)^2 : \lambda \in \Lambda(C_\theta), \Re \lambda \leq 0\}, \\				 
	\label{eq:psa_set}
	\aroots &\coloneqq
	\{ \theta : \afn(\theta) = 0, \, \theta \in (-\pi,\pi] \},	
	\end{align}
\eseq
where $\Arg : \C \setminus \{0\} \to (-\pi,\pi]$ is the principal value argument function,
the matrix $C_\theta$ is defined in~\eqref{eq:eigCD},
and the term $\Arg(-\imagunit \lambda)$ in~\eqref{eq:psa_fn} is squared 
in order to smooth its value out when transitioning to/from zero.
We explain this in more detail later on, but the squaring is done in order
to make $\afn$ easier to approximate globally on its domain.
Note that the definition of~$\afn$
excludes eigenvalues in the open right half of the
complex plane since the spectrum of $C_\theta$ is symmetric with respect to the imaginary axis.

% Due to the [], need to define this separately to pass to \begin{lemma}[string]
\def\citeThmsThree{cf. \cite[Theorems~2.7, 3.5, and 4.6]{Mit21}}
\begin{theorem}[Properties of $\afn$; \citeThmsThree]
Let $A \in \C^{m \times m}$, $\eps \geq 0$, and~$z_0 \in \C$ 
be such that $\eps$ is not a singular value of $A  - z_0I$.
Then, the function $\afn$ defined in~\eqref{eq:psa_fn} has the following properties:
\begin{enumerate}[leftmargin=*,label=\normalfont(\roman*)]
\item $\afn(\theta) \geq 0$ on its entire domain, i.e., $\forall \theta \in (-\pi,\pi]$,
\item $\afn(\theta) = 0 \quad \Longleftrightarrow \quad \exists r > 0$ such that $\imagunit r \in \Lambda(C_\theta) 
	\quad \Longleftrightarrow \quad \rayth \cap \ps(A) \neq \varnothing$,
\item $\afn$ is continuous on its entire domain, 
\item $\afn$ is differentiable at a point $\theta$ if the eigenvalue $\lambda \in \Lambda(C_\theta)$ attaining the value of $\afn(\theta)$ is unique and simple.
\end{enumerate}
Furthermore,  the following properties hold for the associated set $\aroots$ defined in \eqref{eq:psa_set}:
\begin{enumerate}[leftmargin=*,label=\normalfont(\roman*),resume]
\item $\eps = 0
	\quad \Longleftrightarrow \quad
	\mu(\aroots) = 0$, 
\item $\eps_1 < \eps_2
	\quad \Longleftrightarrow \quad
	\mu (\aroots[\eps_1]) < \mu (\aroots[\eps_2])$,
\item if $\eps > f_A(0,\theta)$ for any $\theta \in \R$, then $\mu(\aroots) = 2\pi$,
\item $\aroots$ can have up to $m$ connected components.
\end{enumerate}
\label{thm:props_ps}
\end{theorem}
\begin{proof}
Noting that $-\imagunit \lambda$ in \eqref{eq:psa_fn} is always in the (closed)
upper half of the complex plane, 
statements~(i) and (ii) hold by the definition of $\afn$ and \cref{cor:ray_thru}. 
Statement~(iii) follows from 
the continuity of eigenvalues and our assumption that $\eps$ is not a singular value of $A-z_0I$,
equivalently \mbox{$\eps^2 \not\in \Lambda((A  - z_0I)(A - z_0I)^*)$}, and thus, by \cref{lem:zero_eig}, $0 \not\in \Lambda(C_\theta)$ is ensured for any~$\theta$.
Statement~(iv) follows from standard perturbation theory for simple eigenvalues and by the definition of~$\afn$.

Now turning to $\aroots$, 
either $z_0 \in \interior \ps(A)$ or $z_0 \not\in \ps(A)$ must hold since 
our assumption on~$\eps$ precludes $z_0$ from being a boundary point.
If $\eps > f_A(0,\theta)$, then $z_0 \in \interior \ps(A)$, which in turn implies
that $\rayth \cap \ps(A) \neq \varnothing$ for all $\theta$,
thus proving (vii).
Now assume $z_0 \not\in \ps(A)$.
Statement (viii) is a consequence of the well-known fact that 
for any matrix~$A \in \C^{m \times m}$, its $\eps$-pseudospectrum has at most $m$ connected components.
For any component~$\mathcal{G}$ of $\ps(A)$, by connectedness and (ii),
it is clear that $\mathcal{G}$ is associated with a single interval $\mathcal{I} \subseteq (-\pi,\pi]$
such that $\afn(\theta) = 0$ if and only if $\rayth \cap \ps(A) \neq \varnothing$.
Since $\aroots$ is simply the union of those intervals associated with the components of $\ps(A)$,
of which there can be at most $m$, $\aroots$ also has at most $m$ components, thus proving (viii).
Statement (vi) follows by noting that $\ps[\eps_1](A) \subset \ps[\eps_2](A)$ is equivalent to 
$\aroots[\eps_1] \subset \aroots[\eps_2]$.
Since $\rayth \cap \ps[\eps_1](A) \neq \varnothing$ implies that
\mbox{$\rayth \cap \ps[\eps_2](A) \neq \varnothing$}, 
it follows that $\afn[\eps_1](\theta) = 0$ implies $\afn[\eps_2](\theta) = 0$, and so 
$\aroots[\eps_1] \subset \aroots[\eps_2]$.
Now suppose that $\aroots[\eps_1] \supset \aroots[\eps_2]$ and let $\theta \in \aroots[\eps_1] \setminus \aroots[\eps_2]$;
hence $\afn[\eps_1](\theta) = 0$ but $\afn[\eps_2](\theta) > 0$.  Then $\rayth$ intersects~$\ps[\eps_1](A)$ but not $\ps[\eps_2](A)$,
and so $\ps[\eps_1](A) \subset \ps[\eps_2](A)$ cannot hold, a contradiction.
Finally, for (v),
if $\eps = 0$, $\ps(A) = \Lambda(A)$, and so $\rayth \cap \ps(A) \neq \varnothing$
for at most $m$ different angles.
As there can be at most $m$ connected components of $\aroots$, if $\mu(\aroots) = 0$ holds, then~$\eps=0$.
\hfill
\end{proof}

Per \cref{thm:props_ps}, $\afn$ is a continuous function and 
$\afn(\theta) = 0$ if and only if \mbox{$\rayth \cap \ps(A) \neq \varnothing$}.
\emph{Thus, by finding roots of $\afn$, we find rays which intersect the $\eps$-pseudospectrum of $A$},
our first step toward finding regions where $\ps(A)$ and $\ps(B$) overlap.
For an illustration of this correspondence, see \cref{fig:comps_ex}, 
where~$\bfn$, the analogue of $\afn$ for matrix $B$, is also plotted.

The properties of $\afn$ listed in \cref{thm:props_ps} show that it is reasonably well behaved.
Satisfying the assumption that $\eps$ is not a singular value of $A-z_0I$ can be trivially met,
e.g., just by choosing~$z_0$ with a bit of randomness.
The dominant cost of evaluating $\afn$ at a point $\theta$ is computing the 
spectrum of~$C_\theta$, i.e., $\bigO(m^3)$ work.
Relative to Gu and Overton's $\seplamD$ algorithm, this is a negligible cost.  
To find the roots of $\afn$, 
we can approximate $\afn$ using Chebfun,
which is why we defined $\afn$ using the squared term $\Arg(-\imagunit \lambda)^2$ instead of just $\Arg(-\imagunit \lambda)$.
As will be made clear in~\cref{sec:support},~$\afn$ transitioning to/from zero corresponds to two (or possibly more non-generically) 
eigenvalues of $C_\theta$ coalescing on the positive portion of the imaginary axis.
Without this squaring, $\afn$ would generally be non-Lipschitz at such transition points 
and thus it could be difficult and/or expensive to approximate via interpolation; 
the squaring smooths out this high rate of change so that $\afn$ is easier to approximate.
Although the analogues of $a_\eps$ and \cref{thm:props_ps} 
that appeared in \cite{Mit21} for computing Kreiss constants and the distance to uncontrollability
were sufficient to develop interpolation-based globality certificates for those quantities,
for $\seplamD$, $a_\eps$ and \cref{thm:props_ps} are insufficient.

\def\imgscale{0.422}

\begin{figure}[t]
\centering
\subfloat[$\ps(A)$ and $\ps(B)$]{
\includegraphics[scale=\imgscale,trim={0.0cm 0.0cm 0.0cm 0.0cm},clip]{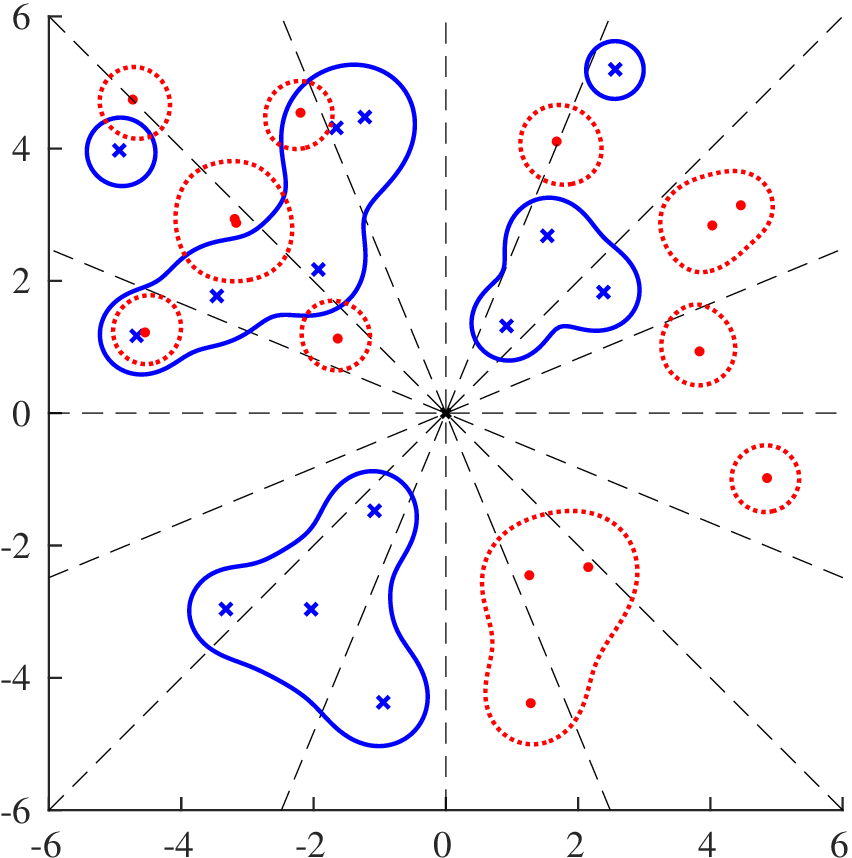} 
} 
\subfloat[$\afn(\theta)$, $\bfn(\theta)$, and $\olfn(\theta)$]{
\includegraphics[scale=\imgscale,trim={0.0cm 0.0cm 0.0cm 0.0cm},clip]{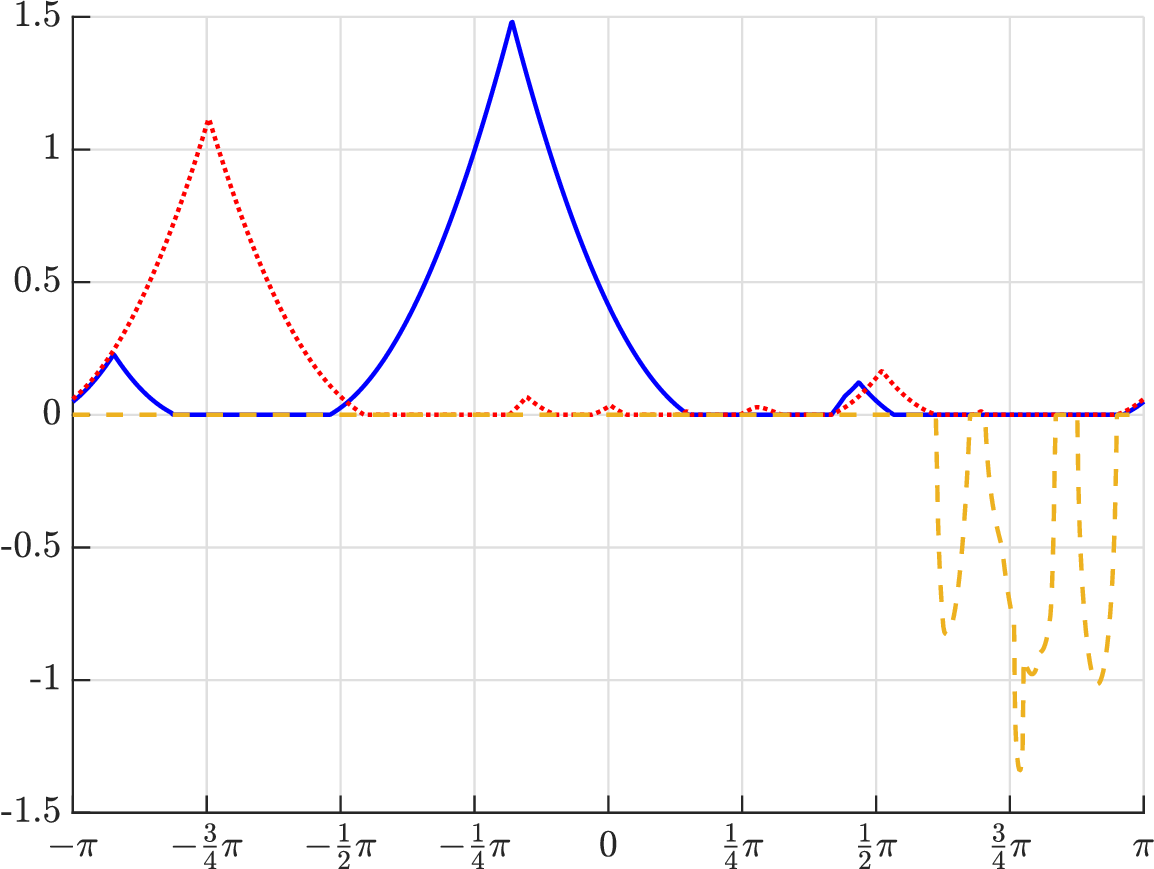} 
}
\caption{For two randomly generated matrices $A,B \in \C^{14 \times 14}$,
the left pane shows their eigenvalues (respectively x's and dots), and
$\ps(A)$ and $\ps(B)$ (respectively solid and dotted contours)
 for~\mbox{$\eps = 0.3 > \seplamD$}.
The search point $z_0$ is the origin; rays emanating from it are depicted by dashed lines.  
The right pane shows corresponding plots of $\afn$, $\bfn$ 
(respectively solid and dotted curves), and~$\olfn$ (dashed),
where $\afn$ is defined in \eqref{eq:psa_fn}, $\bfn$ is its analogue for matrix $B$,
and~$\olfn$ is defined in \eqref{eq:suff_fn}.
On the left, 
rays in the lower left quadrant only intersect $\ps(A)$ or neither $\eps$-pseudospectrum,
while rays in the lower right quadrant only intersect $\ps(B)$ or neither $\eps$-pseudospectrum.
Correspondingly, for $(-\pi,-\tfrac{1}{2}\pi]$ on the right, we see that 
$\afn$ has zeros but $\bfn$ is always positive, and vice versa for~$(-\tfrac{1}{2}\pi,0]$.
Meanwhile, there exist rays in the upper right quadrant that pass through both 
$\ps(A)$ and $\ps(B)$, but $\ps(A)$ and $\ps(B)$ do not overlap in this region;
thus, on the right for $(0,\tfrac{1}{2}\pi]$, we see that $\afn$ and $\bfn$ do have zeros in common,
but $\olfn$ is equal to zero on this interval.
Finally, in the upper left quadrant, $\interior \ps(A)$ and $\interior \ps(B)$ 
do in fact overlap, and so on the right, we see that $\afn$ and $\bfn$ have zeros in common 
and $\olfn$ is indeed negative on a subset of $(\tfrac{1}{2}\pi,\pi]$ with positive measure.
}
\label{fig:comps_ex}
\end{figure}

\section{Locating pseudospectral overlap}
\label{sec:overlap}
As part of locating regions where $\ps(A) \cap \ps(B) \neq \varnothing$,
we will also need to locate the components of $\ps(B)$ with respect to the same ``search point" $z_0$ and given value of~$\eps$.
Thus, for matrix $B$, let $f_B$ and $F_B$ respectively denote the analogues of $f_A$ and $F_A$ defined in~\eqref{eq:fA},
and similarly, let~$\bfn$ and $\broots$ be respective analogues of $\afn$ and $\aroots$ defined in \eqref{eq:psa_comps}.
For matrix~$A$, we continue to use~$C_\theta$ to denote its associated Hamiltonian matrix defined in~\eqref{eq:eigCD},
while we use~$S_\theta$ to denote the analogue Hamiltonian matrix for~$B$, as both matrices will be needed.
Per the assumption of~\cref{thm:props_ps}, we now need to assume that~$\eps$ is not a singular value of either $A - z_0I$ or~$B - z_0I$,
which again, can be easily satisfied by choosing $z_0$ with some randomness.
In establishing tools for locating 
pseudospectral overlap, we will make use of the following elementary result.

\begin{lemma}
Let $\setA,\setB \subset \R$ be such that $\setA$ and $\setB$
respectively consist of $m$ and $n$ connected components.
Then $\setA \cap \setB$ 
can have up to $m + n - 1$ connected components.
\label[lemma]{lem:num_comp}
\end{lemma} 
\begin{proof}
Let $\setA = \setA_1 \cup \cdots \cup \setA_m$, where
each $\setA_j$ is a connected component of $\setA$ and $\setA_j \cap \setA_k = \varnothing$
for all~$j \neq k$, and in an analogous fashion, let \mbox{$\setB = \setB_1 \cup \cdots \cup \setB_n$}.
Without loss of generality, assume that $m \leq n$.
If $m=1$, suppose that the claim is not true, i.e., that $\setA \cap \setB$ has more than~$n$ components.
Then there exists at least one pair of numbers $x$ and $y$ that are in different components 
of~$\setA \cap \setB$ but must be in the same component $\setB_j$ of $\setB$.
 However, by connectedness of the components of~$\setA$ and $\setB$, 
 we have that $[x,y] \subset \setA_1 = \setA$ and $[x,y] \subset \setB_j$.
 Therefore $[x,y] \subset \setA \cap \setB$, contradicting that~$x$ and~$y$ are in different components
 of $\setA \cap \setB$.
 For the inductive step, now assume that the claim holds when~$\setA$ consists
 of $j$ components, for $j = 1,\ldots,m-1$ and $j < n$, 
 and suppose that $\setA$ has $m$ components.
Let~$s = \tfrac{1}{2}(a_\mathrm{L} + a_\mathrm{R})$, where 
$a_\mathrm{L} = \sup_{a \in \setA_{m-1}} a$ and \mbox{$a_\mathrm{R} = \inf_{a \in \setA_m} a$},
and define $\setB_\mathrm{L} \coloneqq \{ b : b \in \setB, b < s \}$ and
\mbox{$\setB_\mathrm{R} \coloneqq \{ b : b \in \setB, b > s \}$}.
Clearly $\setB_\mathrm{L}$ and $\setB_\mathrm{R}$ are disjoint 
and~$\setB_\mathrm{L} \cup \setB_\mathrm{R} = \setB \setminus \{s\}$.
Letting $n_\mathrm{L}$ and~$n_\mathrm{R}$ denote the respective number of connected components
of~$\setB_\mathrm{L}$ and $\setB_\mathrm{R}$, it follows that 
$n_\mathrm{L} + n_\mathrm{R} = n$ if~$s \not \in \interior \setB$ and $n_\mathrm{L} + n_\mathrm{R} = n + 1$
otherwise.
Applying the inductive hypothesis, $\{\setA_1 \cup \cdots \cup \setA_{m-1}\} \cup \setB_\mathrm{L}$
has at most $(m-1) + n_\mathrm{L} - 1$ connected components,
while $\setA_m \cup \setB_\mathrm{R}$ has at most $n_\mathrm{R}$ connected components.
Noting that 
$\{\setA_1 \cup \cdots \cup \setA_{m-1}\} \cup \setB_\mathrm{L}$ 
and~$\setA_m \cup \setB_\mathrm{R}$ are also disjoint and their union is~$\setA \cap \setB$, since~$s \not \in \setA \cap \setB$, 
it follows that $\setA \cap \setB$ has at most $(m-1) + n_\mathrm{L} - 1 + n_\mathrm{R} \leq m + n -1$ connected components.
The bound is tight, as
one can construct $\setA$ such that $\setA_j$ intersects both $\setB_j$ and $\setB_{j+1}$ for~$j = 1,\ldots,m-1$,
while $\setA_m$ intersects~$\setB_j$ for~$j=m-1,\ldots,n$.
\hfill
\end{proof}

\begin{theorem}[Properties of $\afn + \bfn$ and a necessary condition for overlap]
Let $A \in \C^{m \times m}$, \mbox{$B \in \C^{n \times n}$}, $\eps \geq 0$, and $z_0 \in \C$ be such that 
$\eps$ is not a singular value of either \mbox{$A - z_0 I$} or $B - z_0 I$, and let $\rayth$ be the ray defined in~\eqref{eq:ray}.
Furthermore, let \mbox{$\abroots \coloneqq \{ \theta \in (-\pi,\pi] : \afn(\theta) + \bfn(\theta) = 0 \}$},
where $\afn$ is defined in \eqref{eq:psa_fn} for $A$ and $\bfn$ is its analogue for $B$.
Then the following statements hold:
\begin{enumerate}[leftmargin=*,label=\normalfont(\roman*)]
\item if $\rayth \cap \ps(A) \cap \ps(B) \neq \varnothing$,
	then $\afn(\theta) + \bfn(\theta) = 0$,
\item if $\afn(\theta) + \bfn(\theta) = 0$, 
	then $\rayth \cap \ps(A) \neq \varnothing$
	and 
	$\rayth \cap \ps(B) \neq \varnothing$,
\item $\afn + \bfn$ is continuous on its entire domain $(-\pi,\pi]$,
\item $\afn + \bfn$ is differentiable at a point $\theta$ if $\afn$ and $\bfn$
	are differentiable at $\theta$,
\item $\abroots$ can have up to $m + n - 1$ connected components.
\end{enumerate}
\label{thm:necc}
\end{theorem}
\begin{proof}
The assumption in (i) implies that $\rayth \cap \ps(A) \neq \varnothing$ and 
\mbox{$\rayth \cap \ps(B) \neq \varnothing$}
hold, and so~\mbox{$\afn(\theta)=0$} and $\bfn(\theta)=0$ by~\cref{thm:props_ps}~(ii).
Statements (ii)--(iv) are direct consequences of~\cref{thm:props_ps}~(ii)--(iv).
For (v), note that $\abroots = \aroots \cap \broots$,
where $\aroots$ is defined in \eqref{eq:psa_set} for $A$ and~$\broots$ is its analogue for $B$.
As $\aroots$ and $\broots$ respectively have up to $m$ and $n$ connected components
by \cref{thm:props_ps}, statement~(v) follows from~\cref{lem:num_comp}.
\hfill
\end{proof}

Given an angle $\theta$, \cref{thm:necc} states that $\afn(\theta) + \bfn(\theta) = 0$ 
is a necessary condition for the pseudospectra $\ps(A)$ and $\ps(B)$ to 
overlap somewhere along the ray $\rayth$, but 
it is easy to see that this is not a sufficient condition for such overlap.
To obtain such a sufficient condition, 
we now define the function $\olfn : (-\pi,\pi] \to (-\infty,0]$ and an associated set:
\bseq
	\begin{align}
	\label{eq:suff_fn}
	\olfn(\theta) &\coloneqq -\mu \left(\rayth \cap \ps(A) \cap \ps(B) \right), \\
	\label{eq:suff_set}
	\olset &\coloneqq \{ \theta \in (-\pi,\pi] :  \olfn(\theta) < 0 \}.
	\end{align}
\eseq
As $\olset$ is open, it is measurable, and 
via \cref{lem:eig_ham}, we know that the ray $\rayth$ 
can intersect at most~$2m$ and~$2n$ boundary points, respectively, of $\ps(A)$ and $\ps(B)$.
Thus, the number of connected components of $\rayth \cap \ps(A)$ is finite,
as is the number of connected components of~\mbox{$\rayth \cap \ps(B)$};
hence, the intersection in the definition of $\olfn$ is measurable.
Moreover, \cref{lem:eig_ham} allows us to determine these intervals (or isolated points), and so the value of $\olfn(\theta)$
can be computed simply by calculating how much the intervals of 
$\rayth \cap \ps(A)$ overlap those of $\rayth \cap \ps(B)$;
we explain exactly how this is done in~\cref{sec:implement}.
In addition to $\afn$ and $\bfn$, function $\olfn$ is also plotted in~\cref{fig:comps_ex}.

\begin{theorem}[Properties of $\olfn$ and a sufficient condition for overlap]
Let $A \in \C^{m \times m}$, \mbox{$B \in \C^{n \times n}$}, $\eps \geq 0$, $z_0 \in \C$, $\theta \in \R$, 
and $\rayth$ be the ray defined in \eqref{eq:ray}.
Then for the function $\olfn$ defined in~\eqref{eq:suff_fn}, the following statements hold:
\begin{enumerate}[leftmargin=*,label=\normalfont(\roman*)]
\item $\olfn(\theta) < 0
	\quad \Longleftrightarrow \quad
	\rayth \cap \interior \ps(A) \cap \interior \ps(B) \neq \varnothing$,
\item if $\afn(\theta)+\bfn(\theta) > 0$, then $\olfn(\theta) = 0$,
\item $\olfn$ is continuous on its entire domain $(-\pi,\pi]$, 
\item $\olfn$ is differentiable at a point $\theta$ if $\forall r > 0$ such that 
	$z_0 + r\eit \in \bd \ps(A)$, 
	$\imagunit r$ is a simple eigenvalue of $C_\theta$, 
	and $\forall r > 0$ such 
	that $z_0 + r\eit \in \bd \ps(B)$, 
	$\imagunit r$ is a simple eigenvalue of~$S_\theta$.
\end{enumerate}
Furthermore, the following statements hold for the associated set $\olset$ defined in \eqref{eq:suff_set}:
\begin{enumerate}[leftmargin=*,label=\normalfont(\roman*),resume]
\item $\eps \leq \seplamD 
	\quad \Longleftrightarrow \quad
	\mu(\olset) = 0$,
\item $\seplamD < \eps_1 < \eps_2 
	\quad \Longleftrightarrow \quad
	0 < \mu(\olset[\eps_1]) < \mu(\olset[\eps_2])$,
\item $\min_{\theta \in (-\pi,\pi]}  \olfn(\theta) < 0
	\quad \Longleftrightarrow \quad
	0 < \mu(\olset)
	\quad \Longleftrightarrow \quad
	\seplamD < \eps$.
\end{enumerate}
\label{thm:suff}
\end{theorem}
\begin{proof}
Statement~(i) simply follows from the definition of $\olfn$ given in \eqref{eq:suff_fn}
and noting that the intersection $\rayth \cap \interior \ps(A) \cap \interior \ps(B)$
is either empty or consists of a finite number of open intervals in~$\R$. 
For~(ii), if $\afn(\theta)+\bfn(\theta) > 0$, then  
either $\rayth \cap \ps(A) = \varnothing$ or $\rayth \cap \ps(B) = \varnothing$
holds by~\cref{thm:props_ps}~(ii),
and so \mbox{$\olfn(\theta) = 0$}.
Statement~(iii) follows from the fact the boundaries of $\eps$-pseudospectra vary continuously with respect to~$\eps$,
which is clear from~\eqref{eq:pseudo}, and via~\cref{lem:eig_ham}, do not contain any straight line segments.
Under the assumptions in~(iv), standard perturbation theory for simple eigenvalues applies.

For $\olset$, (v) is a direct consequence of (i) and the definition of $\seplamD$ given in \eqref{eq:min_ed},
as \mbox{$\interior \ps(A) \cap \interior \ps(B) = \varnothing$} if and only if $\eps \leq \seplamD.$
Statement (vi) follows by a similar argument to the proof of~\cref{thm:props_ps}~(vi), 
with $\mu(\olset[\eps_1]) > 0$ if and only if $\eps_1 > \seplamD$ following from~(i).
Statement (vii) is simply a combination of (i) and (vi).
\hfill
\end{proof}

From \cref{thm:suff}~(vii), it is clear that if $\olfn$ can be sufficiently well approximated,
then one can determine whether or not \mbox{$\eps > \seplamD$} holds.
Moreover, as we fully explain in \cref{sec:implement}, via \cref{lem:eig_ham}, knowledge of such angles can be used 
to compute points on the $\eps$-level set of~$\fD$, points which can be used to restart optimization
to find a better (lower) estimate for~$\seplamD$.  
Thus, one may wonder what the point was of considering $\afn + \bfn$ and deriving its associated necessary condition given
in \cref{thm:necc}.
There is in fact a very important reason for this.  

As $\olfn$ is constant (zero) whenever it is not negative,
it can, ironically, be a difficult function to approximate.
The pitfall here is that regions where a function appears to be constant may be \emph{undersampled} by interpolation software,
precisely because the computed estimate of the error on such regions will generally be exactly zero, e.g., 
because the software initially builds a constant interpolant for the region in question.
Thus, there is a concern that approximating $\olfn$ via interpolation may miss regions where~$\olfn(\theta) < 0$ holds,
particularly if these regions are small compared to the regions where $\olfn(\theta) = 0$.
Our solution to this difficulty is to replace $\olfn$ by another non-constant function whenever 
$\olfn(\theta)=0$ holds.
We first consider the continuous function $\tfn : (-\pi,\pi] \to \R$ 
\bseq
	\begin{align}
	\label{eq:cert_simple}
	\tfn(\theta) &\coloneqq
	\begin{cases}
		\afn(\theta) + \bfn(\theta) 
		& \text{if $\afn(\theta) + \bfn(\theta) > 0$} \\
		\olfn(\theta) 			
		& \text{otherwise}
	\end{cases}, 
	\\				 
	\label{eq:troots}
	\troots &\coloneqq \{ \theta \in (-\pi,\pi] : \tfn(\theta) = 0\},
	\end{align}
\eseq
an alternative to approximating $\olfn$;
we have also defined $\mathcal{T}_\eps$, the set of roots of $\tfn$,
as this will be used later.
The key point here is that $\tfn$ tells us at which angles the sufficient condition for 
$\ps(A)$ and $\ps(B)$ to overlap is satisfied ($\tfn(\theta) < 0$), 
where only the necessary condition for overlap is satisfied ($\tfn(\theta) = 0$), or where neither is satisfied ($\tfn(\theta) > 0$).
However, in light of \cref{thm:necc,thm:suff}, it is clear that $\tfn$ could still contain (potentially large) intervals where it is zero,
and generally, regions where $\tfn(\theta) < 0$ holds will often be found \emph{in between} such regions where 
$\tfn$ is the constant zero.
Thus, there is still cause for concern that approximating $\tfn$ to find regions where it is negative may be difficult.
As such, in the next section we introduce an additional nonnegative function to replace 
the portions of $\tfn$ where it is the constant zero.

\begin{remark}
Recall that we added smoothing in the definitions of $\afn$ and $\bfn$ by squaring the~$\Arg(\cdot)$ terms,
as they otherwise may grow like the square root function when they increase from zero (or vice versa),
behavior which can be difficult and expensive to resolve via interpolation.
While $\olfn$ can also exhibit similar non-Lipschitz behavior when it transitions to being negative (and possibly elsewhere when it is already negative),
we have intentionally not smoothed this term. 
The reason is that once an angle $\theta$ is found such that $\olfn(\theta) < 0$, 
there is no need to continue building an interpolant approximation.
This angle can immediately be used to compute new level-set points to restart optimization and improve (lower)
the current estimate to $\seplamD$.
\end{remark}

\section{Locally supporting rays of pseudospectra and our certificate function $d_\eps$}
\label{sec:support}
In this section, we propose a new function with which we can replace 
the constant-zero portions of~$\tfn$.  
However, we begin with the following general definitions,
which are variations of the concept of 
a \emph{supporting hyperplane} in~$\R^n$ \cite[Chapter~2.5.2]{BoyV04}
specialized to~$\C$,
and a pair of related theoretical results.

\begin{definition}
Given a connected set $\setA \subset \C$, a line $\setL \subset \C$ \emph{supports $\setA$
at a point $z \in \bd(\setA) \cap \setL$} if $\setA$ lies completely in one of the closed half-planes defined by $\setL$.
\label[definition]{def:supp}
\end{definition}

\begin{definition}
Given a set $\setB \subset \C$, a line $\setL \subset \C$ \emph{locally supports $\setB$ at a
point $z \in \bd(\setB) \cap \setL$} if line $\setL$ supports $\setA \cap \mathcal{N}$ at $z$
for some neighborhood $\mathcal{N}$ about $z$,
where $\setA$ is a connected component of $\setB$.
A ray~$\ray$ locally supports $\setB$ at $z \in \bd(\setB) \cap \interior \ray$ if the line $\setL$ containing $\ray$ 
locally supports~$\setB$ at~$z$.
\label[definition]{def:local_supp}
\end{definition}

Note that if $\theta$ is a point where $\afn$ transitions from positive to zero (or vice versa),
this implies that the ray $\rayth$ locally supports $\ps(A)$.
Similarly, if $\theta$ is a point where $\bfn$ transitions from positive to zero (or vice versa),
then $\rayth$ locally supports $\ps(B)$.
Thus, it follows that if $\theta$ is a point where~$\afn+\bfn$ transitions from positive to zero (or vice versa), 
then $\rayth$ locally supports either~$\ps(A)$ or~$\ps(B)$ or both
simultaneously (though not necessarily at the same point).
Also note that if $\olfn$ transitions from zero to negative (or vice versa) at $\theta$, 
then $\rayth$ locally supports~$\ps(A) \cap \ps(B)$.
We now derive necessary conditions based on the eigenvalues of $C_\theta$ and $S_\theta$
for these scenarios.
We first consider the case when $\rayth$ locally supports $\ps(A)$.
Note that \cite[p.~371--373]{BurLO03} also informally touches upon this subject and related issues for 
the specific case of vertical lines. 

\begin{lemma}
Let $A \in \C^{m \times m}$, $\eps \geq 0$, $z_0 \in \C$, $\theta \in \R$,
and $\rayth$ be the ray defined in~\eqref{eq:ray}.
If~$\rayth$ locally supports $\ps(A)$,
then the matrix $C_\theta$ defined in~\eqref{eq:eigCD} 
has $\imagunit \hat r$ with $\hat r > 0$ as a repeated eigenvalue with even algebraic multiplicity.  
\label[lemma]{thm:ps_supp}
\end{lemma}
\begin{proof}
Without loss of generality, assume that $z_0 = 0$ and $\theta = 0$, and 
suppose that $\rayth$ locally supports $\ps(A)$ at $\hat r > 0$.
Thus, $\hat r \in \bd \ps(A)$, and so $\smin(A - \hat r I)=\eps$ 
and $\imagunit \hat r \in \Lambda(C_\theta)$ by \cref{lem:eig_ham}.
By~\cref{def:local_supp}, there exists a neighborhood $\mathcal{N}$ (in the open right half-plane) 
about~$\hat r$ such that $(\ps(A) \cap \mathcal{N}) \setminus\rayth$
is connected.
As $\rayth$ separates $\mathcal{N}$ into 
\mbox{$\mathcal{N}_1 = \{ z \in \mathcal{N} : \Im z > 0\}$} and $\mathcal{N}_2 = \{ z \in \mathcal{N} : \Im z < 0\}$,
either \mbox{$\ps(A) \cap \mathcal{N}_1$} or $\ps(A) \cap \mathcal{N}_2$
must be empty.
Without loss of generality, suppose that $\ps(A) \cap \mathcal{N}_1 = \varnothing$, and now consider how eigenvalue $\imagunit \hat r$
evolves as $\theta$ is varied, i.e., 
$\lambda(\theta) \in \Lambda(C_\theta)$ with $\lambda(0) = \imagunit \hat r$.
By continuity, eigenvalue $\lambda(\theta)$ can either move up or down on the imaginary axis
or it can move off the imaginary axis as the value of $\theta$ is increased from zero.
If it moves along the imaginary axis, then locally, we have that 
$\lambda(\theta) = \imagunit r(\theta)$, where~$r : \R \to \R$ is continuous and~$r(0) = \tilde r$.
Since~$\tilde r > 0$, there exists a $\theta_\mathrm{p} > 0$ such that $r(\theta) > 0$ for all $\theta \in (0,\theta_\mathrm{p})$.
By \cref{lem:eig_ham}, it thus follows that $r(\theta) \eit \in \ps(A)$ for all $\theta \in (0,\theta_\mathrm{p})$,
but this contradicts the assumption that $\ps(A) \cap \mathcal{N}_1$ is empty.
Thus, $\lambda(\theta)$ must move off the imaginary axis as the value of~$\theta$ is increased from zero.
Since the eigenvalues of the Hamiltonian matrix $C_\theta$ are 
symmetric with respect to the imaginary axis, by continuity
at least one pair of eigenvalues (or possibly more pairs non-generically) must coalesce on the imaginary axis at $\imagunit \hat r$  as $\theta \to 0$.
\hfill
\end{proof}

Now consider the case when $\rayth$ locally supports $\ps(A) \cap \ps(B)$,
which can happen at a boundary point of either $\ps(A)$ or $\ps(B)$, 
or a shared boundary point of both.  
Building on \cref{thm:ps_supp}, we have the following result.

\begin{lemma}
Let $A \in \C^{m \times m}$, $B \in \C^{n \times n}$, $\eps \geq 0$, $z_0 \in \C$, $\theta \in \R$, 
and $\rayth$ be the ray defined in~\eqref{eq:ray}.
Furthermore,  for matrix $A$, let $C_\theta$ be the matrix defined in~\eqref{eq:eigCD},
and let~$S_\theta$ be its analogue for matrix~$B$.
If~$\rayth$ locally supports $\ps(A) \cap \ps(B)$ at a point $z \in \C$,
then at least one, and possibly all, of the following conditions must hold:
\begin{enumerate}[leftmargin=*,label=\normalfont(\roman*)]
\item $C_\theta$ and/or $S_\theta$ has $\imagunit \hat r$ with $\hat r > 0$ as a repeated eigenvalue with even algebraic multiplicity,
\item $C_\theta$ and $S_\theta$ have an eigenvalue $\imagunit \hat r$ with $\hat r > 0$ in common.
\end{enumerate}
\label[lemma]{thm:overlap_supp}
\end{lemma}
\begin{proof}
Without loss of generality, we can assume that $z_0 = 0$ and $\theta = 0$, 
and so $z$ is on the positive part of the real axis, i.e., $z = \hat r$ for some $\hat r > 0$.
If $\rayth$ locally supports $\ps(A) \cap \ps(B)$ at~$\hat r$,
either~$\hat r \in \bd \ps(A)$ but not $\bd \ps(B)$ (or vice versa) or
$\hat r$ is a shared boundary point of both~$\ps(A)$ and~$\ps(B)$.
If $\hat r$ is not a shared boundary point, 
then  $\rayth$ must locally support either~$\ps(A)$ or~$\ps(B)$ at~$\hat r$,
and so \cref{thm:ps_supp} applies, yielding the ``or" part of (i).
Now suppose~$\hat r$ is a shared boundary point, and so $\smin(A - \hat r I) = \smin(B - \hat r I) = \eps$.
Then by \cref{lem:eig_ham}, $\imagunit \hat r$ is an eigenvalue of 
both $C_\theta$ and $S_\theta$, yielding~(ii).
Furthermore, 
$\rayth$ may or may not also locally support $\ps(A)$ and/or $\ps(B)$ at $\hat r$.
All four scenarios are possible, with the ``and" part of (i) corresponding to when 
the ray simultaneously locally supports both $\ps(A)$ and $\ps(B)$ at $\hat r$.
\hfill
\end{proof}

Recall the set of roots $\troots$ of $\tfn$, which is defined in~\eqref{eq:troots}.
If $\theta \in \troots$, then the necessary condition for overlap $\afn(\theta) + \bfn(\theta) = 0$ is satisfied,
and so $\rayth$ intersects both $\ps(A)$ and $\ps(B)$.
However, as~$\olfn(\theta) = 0$, the sufficient condition is not met, and via \cref{thm:suff}, it follows 
that~$\ps(A)$ and~$\ps(B)$ either have no points in common along $\rayth$,
or at most only boundary points in common.
For a function to replace the regions of $\tfn$ where $\tfn(\theta) =0$, i.e., $\troots$,
we propose a function $d_\eps^{AB} : \troots \to [0,\infty)$ that
is a measure of how close $\ps(A)$ and $\ps(B)$ are to sharing a boundary point
along $\rayth$.
To that end, let
\bseq	
	\label{eq:dAB}
	\begin{align}
	d_\eps^{AB}(\theta) &\coloneqq \min \{ d_\eps^A(\theta), \, d_\eps^B(\theta) \}, \qquad \text{where} \\
	d_\eps^A(\theta) &\coloneqq
	\min \{ f_A(r,\theta) - \eps : z_0 + r\eit \in\rayth \cap \bd \ps(B) \}, \\
	d_\eps^B(\theta) &\coloneqq
	\min \{ f_B(r,\theta) - \eps : z_0 + r\eit \in\rayth \cap \bd \ps(A)\},
	\end{align}
\eseq
where $f_A$ is defined in \eqref{eq:fA} for matrix $A$ and $f_B$ is its analogue for matrix $B$.
Since $\theta \in \troots$, both~$\rayth \cap \bd \ps(A)$ and $\rayth \cap \bd \ps(B)$ must 
be nonempty, and so the functions are well defined.
The purpose of $d_\eps^A$ is to provide a nonnegative measure of how close $\ps(B)$ is to 
touching~$\ps(A)$ along the given ray~$\rayth$, and vice versa for $d_\eps^B$.
Note that if \mbox{$\rayth \cap \bd \ps(A) \cap \bd \ps(B) \neq \varnothing$},
then~$d_\eps^A(\theta) = d_\eps^B(\theta) = 0$, but otherwise 
$d_\eps^A(\theta)$ and $d_\eps^B(\theta)$ are typically not the same value. 
While technically $d_\eps^A$ alone (or $d_\eps^B$) would suffice as a closeness measure
of the two pseudospectra along a given ray, 
we have observed that their pointwise minimum, i.e., $d_\eps^{AB}$, 
is often cheaper to approximate.  
Important properties of $d_\eps^{AB}$ are summarized in the following statement.

\begin{theorem}[Properties of $d_\eps^{AB}$]
Let $A \in \C^{m \times m}$, $B \in \C^{n \times n}$, $\eps \geq 0$, and $z_0 \in \C$ be such that 
$\eps \geq 0$ is not a singular value of either $A - z_0 I$ or $B - z_0 I$, and let $\rayth$ be the ray defined in \eqref{eq:ray}.
Furthermore, let $d_\eps^{AB}$ be as defined in \eqref{eq:dAB} on domain 
$\troots$ defined in \eqref{eq:troots}.
Then for any point $\theta \in \troots$, the following statements hold:
\begin{enumerate}[leftmargin=*,label=\normalfont(\roman*)]
\item $d_\eps^{AB}(\theta) \geq 0$,
\item $d_\eps^{AB}(\theta)=0 
		\quad \Longleftrightarrow \quad
		\rayth \cap \ps(A) \cap \ps(B) \neq \varnothing
		\quad \Longleftrightarrow \quad 
		\rayth \cap \bd \ps(A) \cap \bd \ps(B) \neq \varnothing$,
\item $d_\eps^{AB}$ is continuous at $\theta$ if every eigenvalue $\imagunit r$,
	of either $C_\theta$ or $S_\theta$, that attains the minimum in $d_\eps^{AB}(\theta)$ is simple,
\item $d_\eps^{AB}$ is differentiable at $\theta$ if there are no ties for $d_\eps^{AB}(\theta)$, i.e.,
	it is attained via $f_A(r,\theta)$ or $f_B(r,\theta)$ but not both,
	the corresponding minimum singular value is simple, and 
	there is a single eigenvalue~$\imagunit r$, of either $C_\theta$ or $S_\theta$ as appropriate, 
	that attains $d_\eps^{AB}(\theta)$, where this eigenvalue is simple.
\end{enumerate}
\label{thm:dAB}
\end{theorem}
\begin{proof}
Statements (i) and (ii) are simple but important direct consequences of the definition of~$d_\eps^{AB}$ 
and the fact that its domain is restricted to $\troots$,
since otherwise $d_\eps^{AB}(\theta)$ could be negative (or undefined) for some~$\theta$
and the equivalences in (ii) would not hold.
For statement (iii), consider~$d_\eps^B(\theta)$ and recall that by \cref{lem:eig_ham}, 
\mbox{$z_0 + \hat r\eit \in\rayth \cap \bd \ps(A)$} 
is always associated with an eigenvalue~$\imagunit \hat r$ of $C_\theta$.
Since eigenvalues are continuous, 
eigenvalue $\imagunit \hat r$ can either move continuously along the positive portion of the imaginary
axis or leave this region as $\theta$ is varied.  
Clearly, the former case cannot cause a discontinuity in $d_\eps^B$, so consider the latter.
By the assumption on $\eps$, zero can never be an eigenvalue of $C_\theta$ for any $\theta$, 
and clearly the eigenvalues of a matrix are all finite.
Thus, if an eigenvalue leaves the positive portion of the imaginary axis, it cannot be by going through 
the origin or infinity.
Since the eigenvalues of the Hamilton matrix~$C_\theta$ 
are symmetric with respect to the imaginary axis, 
a simple eigenvalue cannot leave the imaginary axis, 
and a repeated eigenvalue is excluded by assumption; hence, $d_\eps^B$ must be continuous at $\theta$.
The same argument shows that  $d_\eps^A$ is continuous at $\theta$ under the 
analogous assumptions for the eigenvalues of $S_\theta$, and 
so $d_\eps^{AB}$ is continuous at $\theta$.
For (iv), the assumptions mean that there are no ties for the $\min$ functions and 
standard perturbation theory for simple singular values and simple eigenvalues applies. 
\hfill
\end{proof}

While \Cref{thm:dAB} verifies that $d_\eps^{AB}$ is reasonably well behaved,
$d_\eps^{AB}$ may have jump discontinuities. 
However, $d_\eps^{AB}$ is discontinuous at point $\theta \in \interior \troots$
only if two conditions simultaneously hold: 
$\rayth$ locally supports $\ps(A)$ or $\ps(B)$ 
at a point $z_0 + \hat r\eit$ with $\hat r > 0$, and this value $\hat r$ is the one that attains the value of $d_\eps^{AB}(\theta)$.
As a result, we expect such discontinuities to be relatively few, and so this should not be a problem in practice.
Functions $d_\eps^A$ and $d_\eps^B$ typically do not have non-Lipschitz behavior when they transitions to/from 
zero, and so we have not added smoothing when them and defining~$d_\eps^{AB}$.  
When~$\fD$ has a unique minimizer, $d_\eps^{AB}$ only has a single root for~$\eps = \seplamD$.

\def\imgscale{0.422}

\begin{figure}[t]
\centering
\subfloat[$\ps(A)$ and $\ps(B)$]{
\includegraphics[scale=\imgscale,trim={0.0cm 0.0cm 0.0cm 0.0cm},clip]{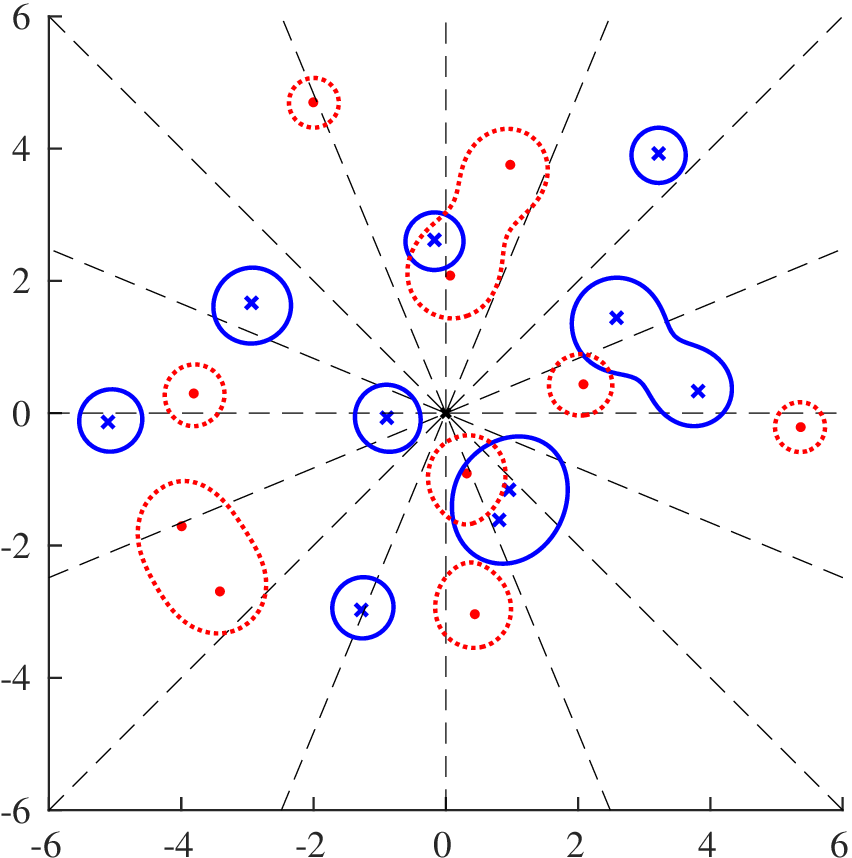} 
} 
\subfloat[$d_\eps(\theta)$]{
\includegraphics[scale=\imgscale,trim={0.0cm 0.0cm 0.0cm 0.0cm},clip]{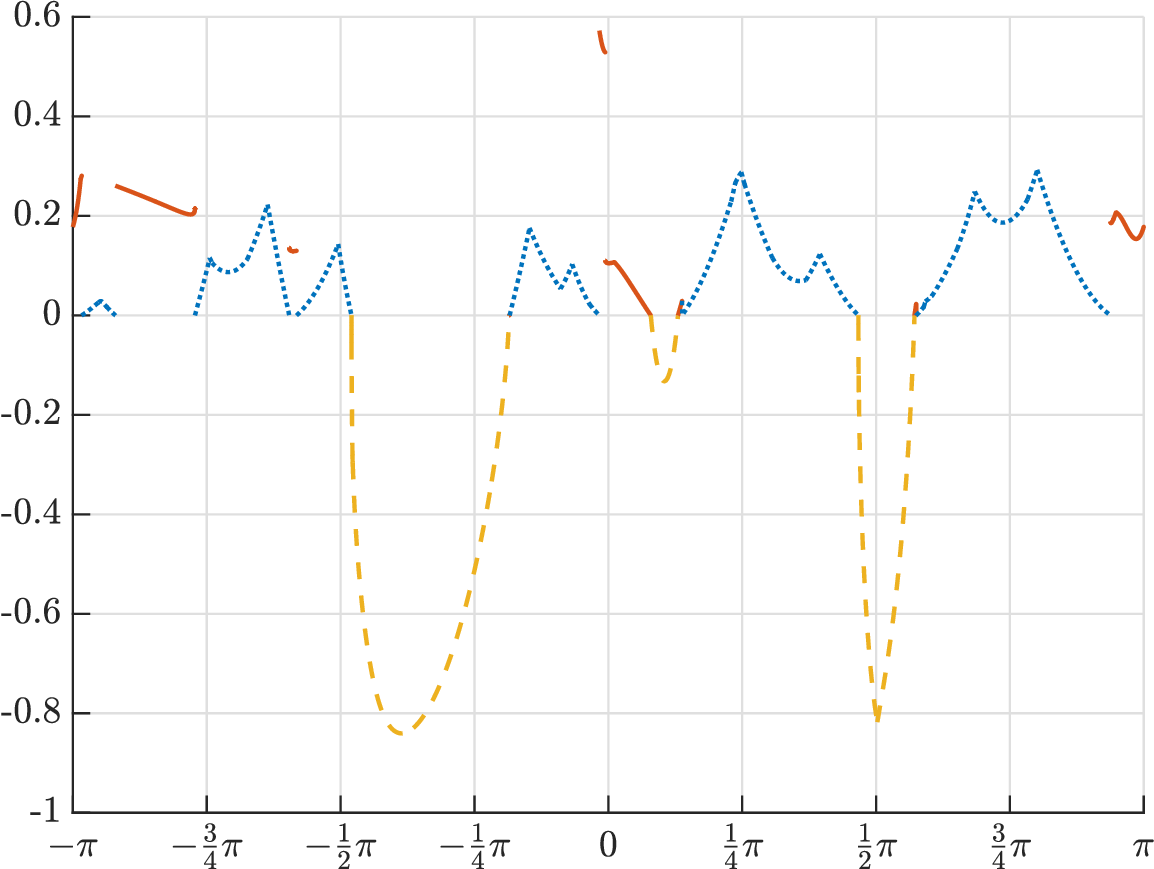} 
}
\caption{
For two randomly generated matrices $A,B \in \C^{10 \times 10}$,
the left pane shows their eigenvalues (respectively x's and dots), and
$\ps(A)$ and $\ps(B)$ (respectively solid and dotted contours)
 for \mbox{$\eps = 0.3 > \seplamD$}.
The search point $z_0$ is the origin; rays emanating from it are depicted by dashed lines.  
The right pane shows a corresponding plot of $d_\eps$, where its components 
are plotted as follows: $\afn + \bfn$ (dotted), 
$\olfn$ (dashed), and $d_\eps^{AB}$ (solid).
For $\theta = -\tfrac{1}{2}\pi$, it can be seen in the left pane that 
$\rayth$ only passes through $\ps(B)$ and 
so $d_\eps(\theta) = \afn(\theta) + \bfn(\theta) > 0$ in the right pane.
Meanwhile for $\theta = \tfrac{1}{2}\pi$, $\rayth$ passes through 
$\interior \ps(A) \cap \interior \ps(B)$ and so $d_\eps(\theta) = \olfn(\theta) < 0$. 
Finally, for $\theta = 0$, while $\rayth$ passes through both $\ps(A)$ and $\ps(B)$,
it never does so simultaneously, 
hence~$\afn(\theta) + \bfn(\theta) = \olfn(\theta) = 0$ and $d_\eps(\theta) = d_\eps^{AB}(\theta) > 0$.
}
\label{fig:cert_fn_ex}
\end{figure}

Combining our three constituent pieces, we now define 
\mbox{$d_\eps : (-\pi,\pi] \to \R$},
our key function for our interpolation-based globality certificate for $\seplamD$:
\beq	
	\label{eq:cert_demmel}
	d_\eps(\theta) \coloneqq
	\begin{cases}
		\afn(\theta) + \bfn(\theta) 
		& \text{if $\afn(\theta) + \bfn(\theta) > 0$}, \\
		\olfn(\theta) 			
		& \text{if $\olfn(\theta) < 0$}, \\
	 	d_\eps^{AB}(\theta) 
		& \text{otherwise.}
	\end{cases}
\eeq
In \cref{fig:cert_fn_ex}, we plot $d_\eps$ for a sample problem 
with \mbox{$\eps > \seplamD$}
in order to illustrate the different components of $d_\eps$.
Recalling that $\afn + \bfn$ is a nonnegative function and so is $d_\eps^{AB}$ on its domain, we immediately have
the following global convergence conditions as a corollary of \cref{thm:necc,thm:suff,thm:dAB}.

\begin{corollary}[Global convergence for $\seplamD$ via $d_\eps$]
Let $A \in \C^{m \times m}$, $B \in \C^{n \times n}$, $\eps \geq 0$, and~$z_0 \in \C$ be such that 
$\eps$ is not a singular value of either $A - z_0 I$ or $B - z_0 I$,
and let $d_\eps$ be the function defined in \eqref{eq:cert_demmel}.
Then 
\[ 
	\min_{\theta \in (-\pi,\pi]} d_\eps(\theta) < 0
	\quad \Longleftrightarrow \quad
	\mu ( \{\theta \in (-\pi,\pi] : d_\eps(\theta) < 0 \}) > 0
	\quad \Longleftrightarrow \quad
	\eps > \seplamD.
\]
\label[corollary]{cor:demmel_global}
\end{corollary}

In the process of devising $d_\eps$, we considered many different possibilities
but found that these  alternatives
were  significantly more expensive to use than $d_\eps$,
even if they had fewer jumps or even none.
For example, we considered an entirely continuous alternative to~$d_\eps$ that replaced its $d_\eps^{AB}$ portions
with a continuous measure of the distance to any of the necessary conditions in \cref{thm:overlap_supp} holding.
However, this function often had more complicated behavior and many many roots than~$d_\eps$
because the necessary conditions in~\cref{thm:overlap_supp} hold for any $\theta$ such that~$\rayth$ locally supports
either of the two pseudospectra or their intersection, and possibly at other angles as well.  
Even when incorporating smoothing to address non-Lipschitz behavior at roots,
this alternative was still much more expensive to approximate than $d_\eps$.
We also tried replacing~$d_\eps^{AB}$ 
with~\mbox{$\min \{ \mu (\rayth \cap \ps(A)), \mu(\rayth \cap \ps(B)) \}$}
and other continuous alternatives, although these 
choices still resulted in jumps when combined when used in conjunction 
with~$\afn + \bfn$ and~$\olfn$. 
But these choices were more expensive to approximate than~$d_\eps^{AB}$
because they generally had more complicated behaviors than~$d_\eps^{AB}$, 
e.g., more nonsmooth points, more oscillatory behavior, etc.
Finally, we considered just using the smallest pairwise distance between 
points in $\rayth \cap \bd \ps(A)$ and~$\rayth \cap \bd \ps(B)$.
This is quite similar to $d_\eps^{AB}$ and can have similar discontinuities,
but it too ended up being more expensive to approximate than~$d_\eps^{AB}$.
That all said, none of the alternatives we considered were prohibitively expensive; 
using any of them to compute $\seplamD$ was still much faster
than the method of Gu and Overton, even though they were generally not as fast as our
ultimate choice for $d_\eps$.

\begin{remark}
Another approach to computing $\seplamD$ is 
via
\begin{equation}
	\label{eq:s_theta}
	s(\theta) \coloneqq \min_{r \in \R} s_\theta(r), 
	\quad \text{where} \quad 
	s_\theta(r) \coloneqq \max \{ \smin(F_A(r,\theta)), \, \smin(F_B(r,\theta)) \},
\end{equation}
i.e., $s(\theta)$ is the minimal value $\fD$ takes along the line defined by~$\theta$ and 
passing through some~$z_0 \in \C$.
It is then immediate that
\begin{equation}
	\label{eq:sld_simple}
	\seplamD = \min_{\theta \in [0,\pi)} s(\theta),
\end{equation}
as this simply rewrites~\eqref{eq:min_zd}
in polar coordinates about~$z_0$.
Thus, using Chebfun to approximate~$s$ and then find a global minimizer in $[0,\pi)$
provides another way to obtain~$\seplamD$.  
One drawback of this approach is that for any given~$\theta$, 
evaluating~$s(\theta)$ is much more expensive than evaluating~$d_\eps(\theta)$.
As we explain in detail in the next section, evaluating~$d_\eps(\theta)$ is essentially direct,
since it only requires solving two eigenvalue problems of order $2m$ and $2n$ and this 
is generally the dominant cost.  
Meanwhile, computing~$s(\theta)$ involves finding a global minimizer of~$s_\theta$,
which requires iteration.
Although we can use~\cref{lem:eig_ham} to construct such an iteration, similar to the level-set methods of~\cite{BoyB90,BruS90}
for computing the~$\Hinf$~norm, the resulting algorithm to compute~$s(\theta)$
would generally only be linearly convergent; the key difference
between here and the $\Hinf$-norm setting is that $s_\theta$,
due to being a $\max$ of two $\min$ functions, will generally will be nonsmooth at its minimizers.
Consequently, evaluating $s(\theta)$ would require solving multiple eigenvalue problems
of $2m$ and $2n$.
Another issue is that although $s$ is continuous, it is still nonsmooth, and 
it is generally more expensive for Chebfun to detect nonsmooth points than jumps; 
see~\cite{PacPT09,perTre20}.
Finally, a third downside is that using Chebfun to \emph{precisely} compute a (likely unique) global minimizer
of some function, e.g., $s$, is a significantly more numerically challenging task than what 
we ask of Chebfun inside our algorithm 
using $d_\eps$, i.e., 
to find \emph{any point} where $d_\eps$ is negative, since as we have shown, the set of such points has positive measure
when $\eps > \seplamD$.
Thus, when attempting to compute $\seplamD$ by applying Chebfun to \eqref{eq:sld_simple},
we nevertheless recommend subsequently refining its computed result by applying
local optimization to $\fD$ initialized from the point in the complex plane found by Chebfun.
\label[remark]{rem:sld_simple}
\end{remark}

\section{Implementation and the cost of our method}
\label{sec:implement}
We now discuss how to implement our $\seplamD$ algorithm, which we have done in \matlab,
and describe its overall work complexity.
We give detailed remarks in the following subsections, while
high-level pseudocode is given in \cref{alg:seplamd}.

\begin{algfloat}[t]
\begin{algorithm}[H]
\floatname{algorithm}{Algorithm}
\caption{\small{Interpolation-based Globality Certificate Algorithm for $\seplamD$}}
\label{alg:seplamd}
\begin{algorithmic}[1]
	\REQUIRE{  
		$A \in \C^{m \times m}$, $B \in \C^{n \times n}$, ``search point" $z_0 \in \C$, and $z_\mathrm{init} \in \C$.
		}
	\ENSURE{ 
		$\eps \approx \seplamD$.
		\\ \quad
	}
	
	\WHILE { true } 
		\STATE $\eps \gets $ computed locally/globally minimal value of $\fD$ initialized from $z_\mathrm{init}$
		\STATE \COMMENT{Begin approximating $d_\eps$ to assert convergence or find new starting points}
		\STATE $p_\eps \gets 1$ \COMMENT{Initial guess for polynomial interpolant $p_\eps$ for approximating
			$d_\eps$}
		\WHILE { $p_\eps$ does not sufficiently approximate $d_\eps$ } 
			\STATE $[\theta_1,\ldots,\theta_q] \gets $ new  sample points from $(-\pi,\pi]$ 
			\STATE \COMMENT {If new starting points are detected, restart optimization to lower $\eps$:} 
			\IF { $d_\eps(\theta_j) < 0$ for some $j \in \{1,\ldots,q\}$ } 
				\STATE $z_\mathrm{init} \gets$ a point in $\bd \{ \rayth[\theta_j] \cap \ps(A) \cap \ps(B) \} \setminus \{z_0\}$
				\STATE \textbf{goto line 2} \COMMENT{Restart optimization from $z_\mathrm{init}$}
			\ENDIF
			\STATE \COMMENT {Otherwise, no starting points detected, keep improving $p_\eps$:}
			\STATE 	$p_\eps \gets$ improved polynomial interpolant of 
					$d_\eps$ via $\theta_1,\ldots,\theta_q$
		\ENDWHILE
		\STATE \COMMENT{$p_\eps$ approximates $d_\eps$ well and no new starting points were encountered} 
		\STATE \COMMENT {However, do  a final check before asserting that $d_\eps$ is nonnegative:} 
		\STATE $[\theta_1,\ldots,\theta_q] = \argmin p_\eps(\theta)$
		\IF { $d_\eps(\theta_j) < 0$ for some $j \in \{1,\ldots,q\}$ }
				\STATE $z_\mathrm{init} \gets$ a point in $\bd \{ \rayth[\theta_j] \cap \ps(A) \cap \ps(B) \} \setminus \{z_0\}$
				\STATE \textbf{goto line 2} \COMMENT{Restart optimization from $z_\mathrm{init}$}
		\ELSE
			\RETURN	$\, $ \COMMENT{$p_\eps \approx d_\eps$ and 
					$\quad  \Longrightarrow \quad \eps \approx \seplamD$}
		\ENDIF 
	\ENDWHILE
\end{algorithmic}
\end{algorithm}
\vspace{-0.3cm}
\algnote{
To keep the pseudocode a reasonable length, we make some simplifying assumptions:
optimization converges to local/global minimizers exactly, $z_\mathrm{init}$ computed
in lines 9 and 19, for restarting optimization, is never a stationary point of $\fD$, 
and the~``search point" $z_0$ is such that 
all encountered values of $\eps$ are not singular values of \mbox{$\smin(A-z_0 I)$} and $\smin(B-z_0 I)$,
per the assumptions given in \cref{sec:components} and \cref{sec:overlap}.
Lines 3-15 describe the core of the interpolation-based globality certificate,
where we only give a broad outline of the interpolation process for approximating $d_\eps$;
note that for numerical reasons, 
each certificate should actually be done with~$\tilde \eps = (1 - \tau) \eps$,
where~$\tau \in (0,1)$ is some relative tolerance.
See \cref{sec:d_eps_eval} and \cref{sec:interp_cert} for more implementation details.
}
\end{algfloat}

\subsection{Choosing a search point}
\label{sec:choosing_z0}
Regarding what search point $z_0$ to use, we recommend
the average of all the distinct eigenvalues of $A$ and $B$.
This helps to ensure the whole domain of $d_\eps(\theta)$ is relevant.
Otherwise, if for a given value of $\eps$, $z_0$ is chosen far from the pseudospectra of $A$ and $B$, 
then \mbox{$\afn(\theta) + \bfn(\theta) = 0$} would only hold on a very small subset of $(-\pi,\pi]$,
which in turn would likely make it harder to find the regions where $d_\eps(\theta)$ is negative.
On every round, our code checks that the choice of $z_0$ still satisfies our needed assumptions and perturbs it slightly
if it does not (in practice, we have not observed that this is necessary).
Finally, if the pseudospectra of $A$ and $B$ both have real-axis symmetry,
by choosing $z_0$ on the real axis, it is then only necessary to 
approximate~$d_\eps(\theta)$ on~$[0,\pi]$.

\subsection{Evaluating $d_{\eps}(\theta)$ and its cost}
\label{sec:d_eps_eval}
Given some $\theta$, evaluating $d_\eps(\theta)$ proceeds as follows.
First, the eigenvalues of both $C_\theta$ and $S_\theta$ are computed.
For increased reliability, it is recommended that this be done via a structure-preserving 
eigensolver such as \cite{BenMX98b}.
From these spectra, it is then trivial to calculate the value of $\afn(\theta) + \bfn(\theta)$
via \eqref{eq:psa_fn}.
If $\afn(\theta) + \bfn(\theta) > 0$, then the value of $d_\eps(\theta)$ has been computed.
Otherwise, evaluating $d_\eps(\theta)$ requires the following additional computations,
which begins with obtaining the value of~$\olfn(\theta)$.  
To that end, we compute $\rayth \cap \ps(A)$ and $\rayth \cap \ps(B)$.
Considering the former, we want to determine the values $r > 0$ such that $f_A(r,\theta) = \eps$,
and via \cref{lem:eig_ham}, we have the following sorted list of candidate values $0 = r_0 < r_1 < \ldots < r_q$
that may satisfy this equality, 
where~$\imagunit r_j$ 
for~$j=1,\ldots,q$ are eigenvalues of $C_\theta$ and we have added $r_0=0$.
Then to compute $\rayth \cap \ps(A)$,
we must assert which intervals on $\rayth$, defined by $[r_{j-1},r_j]$ for $j =1,\ldots, q$, are also in $\ps(A)$.  
There are several ways to do this but a simple and robust way
is to just evaluate~$f_A(\hat r_j,\theta)$ for $\hat r_j = 0.5(r_{j-1} + r_j)$ over $j = 1,\ldots,q$;
since $f_A(\hat r_j,\theta) \neq \eps$, the corresponding interval is not in $\rayth \cap \ps(A)$ if and only if $f_A(\hat r_j,\theta) > \eps$.
Note that it does not matter if we have two or more adjacent intervals in our computed version
of~$\rayth \cap \ps(A)$.  An analogous computation yields $\rayth \cap \ps(B)$.
With these two sets computed, calculating the amount of their overlap along the given ray, i.e., $-\olfn(\theta)$,
is straightforward.  
If $\olfn(\theta) < 0$, then the evaluation of $d_\eps(\theta)$ is done and the 
boundary points of $\rayth \cap \ps(A) \cap \ps(B)$ 
have been also been computed, which are used to restart optimization.
However, if $\olfn(\theta)=0$, then finally we must compute $d_\eps^{AB}(\theta)$ in order
to complete the computation of $d(\theta)$,
though this is this is straightforward
to do from the definition of $d_\eps^{AB}(\theta)$ given in~\eqref{eq:dAB}
and the previous computations.

Recalling our assumption that $m \leq n$, evaluating $d_\eps(\theta)$ is $\bigO(n^3)$ work
if done in the following manner.
Computing all of the eigenvalues of $C_\theta$ and $S_\theta$ is $\bigO(n^3)$ work,
and that is all there is to do when~$\afn(\theta) + \bfn(\theta) > 0$.
But when $\afn(\theta) + \bfn(\theta) = 0$, computing $d_\eps(\theta)$ additionally requires 
computing the values of $f_A(r,\theta)$ and $f_B(r,\theta)$ for different values of $r$.
While the number of values of $r$ is often only a handful, in the worst case, 
it can be $\bigO(m+n)$.
Hence, if we were to evaluate this pair of functions by computing SVDs, we would 
exceed the stated $\bigO(n^3)$ work complexity bound by a factor of $n$.  Fortunately,
there is a more efficient option due to Lui for fast plotting of pseudospectra~\cite{Lui97}.
Since $A$ is square, it has a Schur decomposition $A = UTU^*$, where~$U$ is unitary 
and $T$ is triangular, and moreover, since unitary transformations do not alter the pseudospectrum, 
$\ps(A) = \ps(T)$ holds.
The key benefit of this transformation is that at any point $z_0 + r\eit \in \C$, we have that 
$T -(z_0 + r\eit) I$ remains in triangular form,
and so inverse iteration can be done to compute this shifted matrix's minimum singular value 
using backsolves that only require quadratic work as opposed to the usual cubic work
for solving a linear system.
We need only compute and store Schur decompositions of $A$ and $B$ once in an offline phase, 
which is cubic work, and then we can evaluate $f_A(r,\theta)$ and $f_B(r,\theta)$ for any $r$ and $\theta$
in a most $\bigO(n^2)$ work under the mild assumption that inverse iteration converges 
in relatively few steps.\footnote{For more details on the actual inverse-iteration-based algorithm, 
including pseudocode and code examples, see~\cite{Lui97} and~\cite[Chapter~39]{TreE05},
but note that the latter has the following typo: In ``Core EigTool algorithm" \cite[p.~375]{TreE05},
the second to last line should be \texttt{sigmin(j,k) = 1/sqrt(sig);}, not \texttt{sigmin(j,k) = sqrt(sig);}.}
Hence, evaluating $d_\eps(\theta)$ can always be done within $\bigO(n^3)$ work.
In our own experience, we have seen that ten iterations is generally more than sufficient to 
compute $f_A(r,\theta)$ and $f_B(r,\theta)$ accurately to the full precision of the hardware, and 
that this technique is already faster than computing the full SVD for matrices as small as $50 \times 50$.

\subsection{Approximating $d_\eps$ and restarting}
\label{sec:interp_cert}
To approximate $d_\eps$, we use Chebfun, as it is rather adept at approximating 
functions with nonsmooth points and/or discontinuities.  
As Chebfun normally provides groups of points to evaluate simultaneously (line~6 of \cref{alg:seplamd}),
these evaluations of $d_\eps$ can be done in parallel; see \cite[Section~5.2]{Mit21} for more details.
Furthermore, if $d_\eps(\theta) < 0$ for any of current group of points provided by Chebfun,
we immediately halt Chebfun and use the detected boundary points 
of~\mbox{$\rayth \cap \ps(A) \cap \ps(B)$} (except for $z_0$) to restart optimization (lines~7--11 of \cref{alg:seplamd}).
This is accomplished by throwing an error when a point is encountered such that $d_\eps(\theta) < 0$ holds,
which causes Chebfun to be aborted.
By subsequently catching our own thrown error, we can resume our program to restart another round of optimization.

\subsection{Finding minimizers}
\label{sec:find_local_mins}
Like many other optimization-with-restarts algorithms, it will be necessary to use a \emph{monotonic} optimization solver, 
i.e., one that always decreases the objective function on every iteration, which 
is the case for most unconstrained optimization solvers.
Minimizers of $\fD$ will almost always be nonsmooth,
and at best, we can expect linear convergence from a nonsmooth optimization solver.
However, since there are only two real variables, we expect 
the number of iterations needed to converge to be relatively small.
Thus, as evaluating $\fD$ and its gradient is significantly cheaper 
than evaluating~$d_\eps$, and we expect far fewer function evaluations for the former than the latter,
the cost of~\cref{alg:seplamd} will generally not be dominated by the optimization phases.

To find minimizers of~$\fD$ using only gradient information, 
we use GRANSO: GRadient-based Algorithm for Non-Smooth Optimization \cite{granso}.
GRANSO implements the BFGS-SQP nonsmooth optimization algorithm of~\cite{CurMO17}, 
which can handle nonsmooth constraints, but for problems without constraints, 
it reduces to BFGS with the line search of~\cite{LewO13}, a combination which 
Lewis and Overton have studied and advocated as a method for nonsmooth 
optimization.
While there are no convergence results for BFGS for general nonsmooth optimization,
it nevertheless seems to reliably and accurately converge to nonsmooth stationary values.
Indeed, in their concluding remarks~\cite[p.~160]{LewO13}, Lewis and Overton wrote 
``In our experience with functions with bounded sublevel sets, 
BFGS essentially always generates function values converging linearly to a Clarke stationary value, 
with exceptions only in cases that we attribute to the limits of machine precision. 
We speculate that, for some broad class of reasonably well-behaved functions, this behavior is almost sure."
Since $\fD$ is locally Lipschitz as long as $\seplamD > 0$ and has bounded level sets,
we expect that BFGS will also be an efficient and reliable tool in our setting.
For improved theoretical guarantees, one could follow up optimization via BFGS with a phase of 
the gradient sampling algorithm~\cite{BurLO05}, which would ensure convergence to 
nonsmooth stationary values of $\fD$ when~$\seplamD > 0$.
However, for simplicity, we only use BFGS here.

When restarting optimization, 
our certificate may provide many new starting points.
Restarting from just one would give the smallest chance of converging to a global minimizer on this round,
while restarting from them all could be a waste of time, particularly if this ends up just 
returning the same minimizer over and over again.
In practice, one could prioritize them in terms of most promising first and limit the total number used.
On multi-core machines, optimization can be run from multiple starting points in parallel.

\subsection{Terminating the algorithm}
In addition to the convergence tests described in \cref{alg:seplamd},
it is also necessary to terminate the algorithm if consecutive estimates for $\seplamD$ are identical.    
The reason is that we cannot expect optimization solvers to find minimizers \emph{exactly}.
If a global minimizer~$\tilde z$ is obtained only up to some rounding error,
then $\seplamD$ has essentially been computed, but our certificate may still detect that
the algorithm has not \emph{truly} converged to a global minimizer, and in this case, 
the algorithm may try to restart optimization (unsuccessfully).  
This is also part of the reason why the certificates should actually be performed 
with $\tilde \eps = (1 - \tau)\eps$, as described in the note under~\cref{alg:seplamd}.

\subsection{The overall work complexity and using lines instead of rays}
\label{sec:work}
In the worst case, the overall work complexity to perform the interpolation-based globality certificates
is $\bigO(k n^3)$, where $k$ is the total number of function evaluations (over all values of $\eps$ encountered).
As restarts tend to happen quickly, $k$ is roughly equal to the number of evaluations needed to approximate $d_\eps$
when $\eps = \seplamD$, and as we will see in the numerical 
experiments,~$k$ can generally be considered to be like a large constant, although it is influenced by the geometry of the two pseudospectra.

When implementing the algorithm, the definition of $d_\eps$ can be modified so that it considers
lines through~$z_0$ instead of rays emanating from~$z_0$.  
This can be beneficial, since we always get information for the direction~$\theta + \pi$ when considering 
$\rayth$, and so this modified version of $d_\eps$ need only be interpolated on~$[0,\pi]$.  
Function $\afn$ measures the minimum argument of $-\imagunit \lambda$ over each eigenvalue $\lambda$ of $C_\theta$, so when using lines instead of rays, it must also consider the minimum angle with respect to the negative real axis.
These additional angles are computed by simply switching the sign of the imaginary part of each eigenvalue $\lambda$.
The same change is made for $\bfn$, while modifying $\olfn$ and $d_\eps^{AB}$ is straightforward.
While using lines often results in less overall work, this is not always the case,
as it can sometimes make $d_\eps$ more complicated and thus more expensive to approximate.

\section{Algorithms for $\seplamV$}
\label{sec:varah}
We now briefly turn to the problem of computing Varah's $\seplam$.
We first answer whether or not~\cref{alg:seplamd} extends to $\seplamV$
and then propose a different algorithm to compute~$\seplamV$.

\subsection{Does \cref{alg:seplamd} extend to $\seplamV$?}
\label{sec:varah_alg1}
In the construction of function $d_\eps$ for computing $\seplamD$, nowhere have we needed 
that the same value of $\eps$ be used for the pseudospectra of $A$ and $B$.
Thus for Varah's version of $\seplam$, we can analogously define 
\beq	
	\label{eq:cert_varah}
	d_{\eps_1,\eps_2}(\theta) \coloneqq
	\begin{cases}
		\afn[\eps_1](\theta) + \bfn[\eps_2](\theta) 
		& \text{if $\afn[\eps_1](\theta) + \bfn[\eps_2](\theta)  > 0$}, \\
		\olfn[\eps_1,\eps_2](\theta) 			
		& \text{if $\olfn[\eps_1,\eps_2](\theta) < 0$}, \\
		d_{\eps_1,\eps_2}^{AB}(\theta) 
		& \text{otherwise,}
	\end{cases}
\eeq
where 
\begin{align*}
	\olfn[\eps_1,\eps_2](\theta) 
	&\coloneqq -\mu \left(\rayth \cap \ps[\eps_1](A) \cap \ps[\eps_2](B) \right), \\
	d_{\eps_1,\eps_2}^{AB}(\theta) 
	& \coloneqq \min \{d_{\eps_1,\eps_2}^A(\theta), d_{\eps_1,\eps_2}^B(\theta) \}, \\
	d_{\eps_1,\eps_2}^A(\theta) 
	& \coloneqq \min \{ f_A(r,\theta) - \eps_1 :\rayth \cap \bd \ps[\eps_2](B) \}, \\
	d_{\eps_1,\eps_2}^B(\theta) 
	& \coloneqq \min \{ f_B(r,\theta) - \eps_2 :\rayth \cap \bd \ps[\eps_1](A) \},
\end{align*}
and $f_A$ is defined in \eqref{eq:fA} for matrix $A$, while $f_B$ is its analogue for matrix $B$.
Although this will not allow us to compute $\seplamV$ to arbitrary accuracy,
we do have the following necessary condition as another corollary of \cref{thm:suff}.

\begin{corollary}[A necessary condition for $\eps_1 + \eps_2 = \seplamV$ via $d_{\eps_1,\eps_2}$]
Let $A \in \C^{m \times m}$, \mbox{$B \in \C^{n \times n}$}, $\eps_1,\eps_2 \geq 0$, and $z_0 \in \C$ be such that 
$\eps_1$ and $\eps_2$ are, respectively, not singular values of $A - z_0 I$ and $B - z_0 I$,
and let $d_{\eps_1,\eps_2}$ be the function defined in \eqref{eq:cert_varah}.
Then 
\[ 
	\min_{\theta \in (-\pi,\pi]} d_{\eps_1,\eps_2}(\theta) < 0
	\quad \Longleftrightarrow \quad
	\mu ( \{\theta \in (-\pi,\pi] : d_{\eps_1,\eps_2}(\theta) < 0 \}) > 0,	
\]
and 
\[
	\eps_1 + \eps_2 > \seplamV 
	\qquad \text{if} \qquad
	\min_{\theta \in (-\pi,\pi]} d_{\eps_1,\eps_2}(\theta) < 0.
\]
\label[corollary]{cor:varah_necc}
\end{corollary}

As the last statement in~\cref{cor:varah_necc} is not if-and-only-if,
$d_{\eps_1,\eps_2}$ does not allow us compute~$\seplamV$ with guaranteed accuracy.
However, by modifying \cref{alg:seplamd} to instead find minimizers of $\fV$ and use $d_{\eps_1,\eps_2}$,
we can compute locally optimal upper bounds for $\seplamV$ that at least guarantee the necessary condition
$\interior \ps[\eps_1](A) \cap \interior \ps[\eps_2](B) = \varnothing$ is satisfied,
as this is equivalent to $\min_{\theta \in (-\pi,\pi]} d_{\eps_1,\eps_2}(\theta) = 0$.
This is notably better than just computing upper bounds via finding minimizers of $\fV$, since the corresponding 
values of $\eps_1$ and $\eps_2$ associated with minimizers are not guaranteed to satisfy this necessary condition.
However, when either~$\eps_1=0$ or~$\eps_2=0$ holds at the computed minimizer, note that 
$\interior \ps[\eps_1](A) = \varnothing$ or $\interior \ps[\eps_2](B) = \varnothing$ holds, and 
so satisfying the necessary condition does not preclude the possibility that an eigenvalue of $A$ may be
in $\interior \ps[\eps_2](B)$ or vice versa. 
Thus, when approximating $\seplamV$ via this extended algorithm, one should always compute
\beq
	\label{eq:varah_check}
	\tilde \eps = \min \left\{ \min_{\lambda \in \Lambda(B)} \smin(A - \lambda I), \, \min_{\lambda \in \Lambda(A)} \smin(B - \lambda I) \right\},
\eeq
which computes an upper bound $\tilde \eps \geq \seplamV$ such that no eigenvalues of $A$ 
are in the interior of $\interior \ps[\tilde \eps](B)$ and vice versa.
Nevertheless, when optimization finds minimizers where neither $\eps_1$ nor $\eps_2$ is zero, 
then our certificate can be used to restart optimization if the necessary condition does not hold,
and hence obtain a better estimate for $\seplamV$.

\subsection{A different Chebfun-based algorithm to compute $\seplamV$}
\label{sec:varah_alg2}
Given $z_0 \in \C$, let function $v : [0,\pi) \to \R$ be defined as
\begin{equation}
	\label{eq:v_theta}
	v(\theta) \coloneqq \min_{r \in \R} v_\theta(r), 
	\quad \text{where} \quad 
	v_\theta(r) \coloneqq \smin(F_A(r,\theta)) + \smin(F_B(r,\theta)),
\end{equation}
i.e., $v(\theta)$ is the minimal value $\fV$ takes along the line defined by~$\theta$ and 
passing through~$z_0$.  It then immediately follows that
\begin{equation}
	\label{eq:slv_simple}
	\seplamV = \min_{\theta \in [0,\pi)} v(\theta).
\end{equation}
Since $v$ is continuous function defined on a finite interval, 
as in the alternative $\seplamD$ algorithm discussed in~\cref{rem:sld_simple}, we can
consider approximating $v$ with Chebfun in order to solve~\eqref{eq:slv_simple}.

Unfortunately evaluating $v$ for a given $\theta$ is quite difficult, 
as the level-set iteration for finding a global minimizer of $s_\theta$ described in~\cref{rem:sld_simple}
 does not extend to $v_\theta$.  
However, for some~\mbox{$\eps > \seplamV$}, say, $\eps = \fV(z_0)$, we can
easily calculate a finite interval $[r_1,r_2]$ such that~$v_\theta(r) > \eps$ 
must hold for all $r \not\in [r_1,r_2]$.
To do this, we simple apply~\cref{lem:eig_ham} 
to obtain the two extremal points, say, $a_1$ and $a_2$ with $a_1 \leq a_2$, in the $\eps$-level set of 
$\smin(F_A(r,\theta))$ with~$r$ varying and $\theta$ fixed, and then analogously, also obtain the two extremal
level-set points~$b_1$ and~$b_2$ of $\smin(F_B(r,\theta))$ with $b_1 \leq b_2$.
By taking $r_1 = \max\{ a_1,b_1 \}$ and $r_2  = \min\{a_2,b_2\}$, 
we have that any global minimizer of $v_\theta$ must lie in $[r_1,r_2]$,
since by construction, $v_\theta(r) > \eps$ outside this interval.
Thus, to obtain the value of~$v(\theta)$, we simply solve two eigenvalue problems
to obtain~$[r_1,r_2]$ and then apply Chebfun to approximate $v_\theta$ on $[r_1,r_2]$
in order to obtain its globally minimal value.

Using Chebfun to approximate $v$ over $[0,\pi)$, where for each $\theta$, 
the value of $v(\theta)$ is also computed by applying Chebfun to $v_\theta$,
does lead to quite an expensive algorithm, as many evaluations of~\mbox{$\smin(F_A(r,\theta))$}
and \mbox{$\smin(F_B(r,\theta))$} for different values of $\theta$ and $r$ are required.
However, this nested Chebfun-based algorithm nevertheless has the virtue of
being the very first algorithm to compute $\seplamV$, as opposed to just approximating it, 
e.g., within a factor of two by instead computing $\seplamD$.

Regarding the choice of $z_0$, one might be tempted to use a local minimizer of $\fV$,
 but there are pros and cons to doing so.
 On the upside, if $\eps = \fV(z_0)$ is close to $\seplamV$, $v$ 
 likely will be constant (with value $\eps$) on a much of $[0,\pi)$, or all of it if $\eps = \seplamV$,
 precisely because $z_0$ is a minimizer.
 This can greatly reduce the number of function evaluations required by Chebfun,
 but as discussed earlier in~\cref{sec:overlap}, functions with large constant portions 
 can actually cause Chebfun to terminate prematurely.
 As such, we generally recommend that a minimizer of $\fV$ not be used for $z_0$.

Finally, recalling our recommendation at the end of \cref{rem:sld_simple}, 
when computing $\seplamV$ via~\eqref{eq:slv_simple},
we also similarly recommend refining Chebfun's result via subsequently applying local optimization.
The upper bound given in \eqref{eq:varah_check} should also be computed.

\section{Numerical experiments}
\label{sec:experiments}
All experiments were done in \matlab\ R2021a on a computer with 
two Intel Xeon Gold 6130 processors (16 cores each, 32 total) and 192GB of RAM running CentOS Linux~7.
We implemented our new methods using a recent build of Chebfun (commit \texttt{119f9ad}) 
with \texttt{splitting} enabled and \texttt{novectorcheck},
and for simplicity, computed eigenvalues of $C_\theta$ and $S_\theta$
using \texttt{eig} in \matlab;
to account for rounding errors, 
the real part of any computed eigenvalue $\lambda$ was set to zero if~\mbox{$| \Re \lambda | \leq 10^{-8}$}.
For~\cref{alg:seplamd}, we used v1.6.4 of GRANSO with \texttt{opt\_tol=1e-14}
to find local minimizers and used lines instead of rays for our globality certificates, as we observed
that this was usually a bit faster.
We forgo including any parallel processing experiments here, as we have previously validated the large benefits
of using parallelism with our interpolation-based globality certificate approach in~\cite[Section~5.2]{Mit21}.
The codes used to generate the results in this paper are included in the supplementary materials,
and we plan to add robust implementations to ROSTAPACK: RObust STAbility PACKage \cite{rostapack}.

\def\scaledps{0.32}
\def\seplvpscap{$\ps[\eps_1](A(s))$ and $\ps[\eps_2](B(s))$ for $\eps_1 + \eps_2 = \seplamVAB{A(s)}{B(s)}$ and $s=10$ (left), $s = 5$ (middle), and $s=0$ (right).}

\begin{figure}[t]
\centering
\subfloat[$\ps(A(s))$ and $\ps(B(s))$ for $\eps =  \seplamDAB{A(s)}{B(s)}$ and $s=10$ (left), $s = 5$ (middle), and $s=0$ (right).]{
\includegraphics[scale=\scaledps,trim={0.0cm 0.0cm 0.0cm 0.0cm},clip]{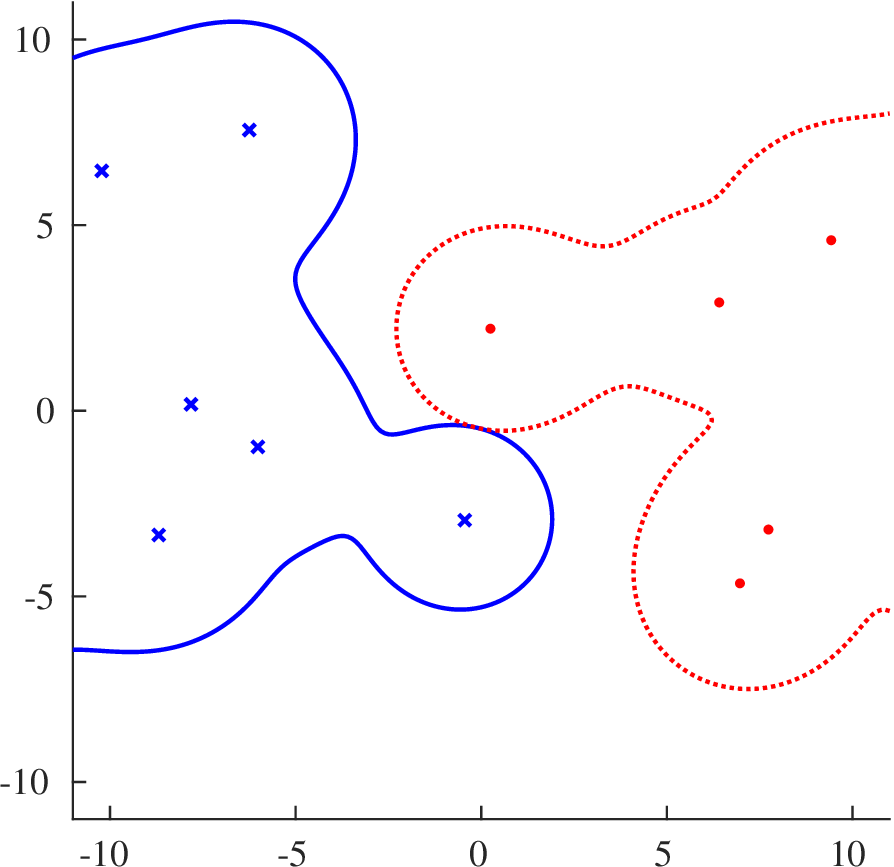} 
\includegraphics[scale=\scaledps,trim={0.cm 0.0cm 0.0cm 0.0cm},clip]{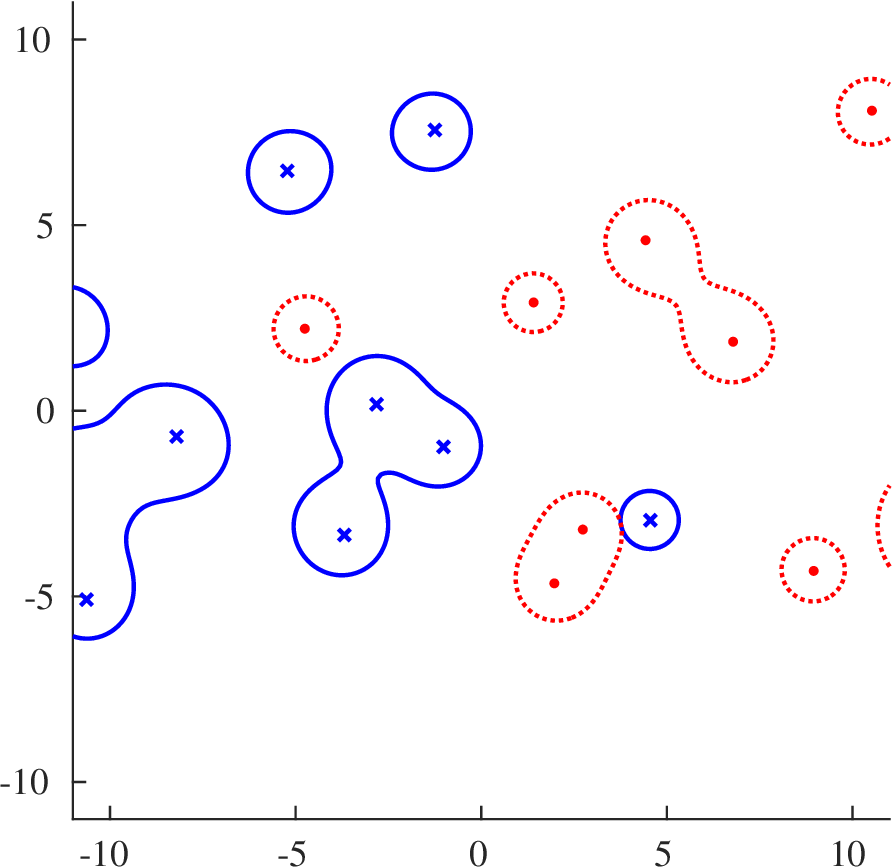} 
\includegraphics[scale=\scaledps,trim={0.0cm 0.0cm 0.0cm 0.0cm},clip]{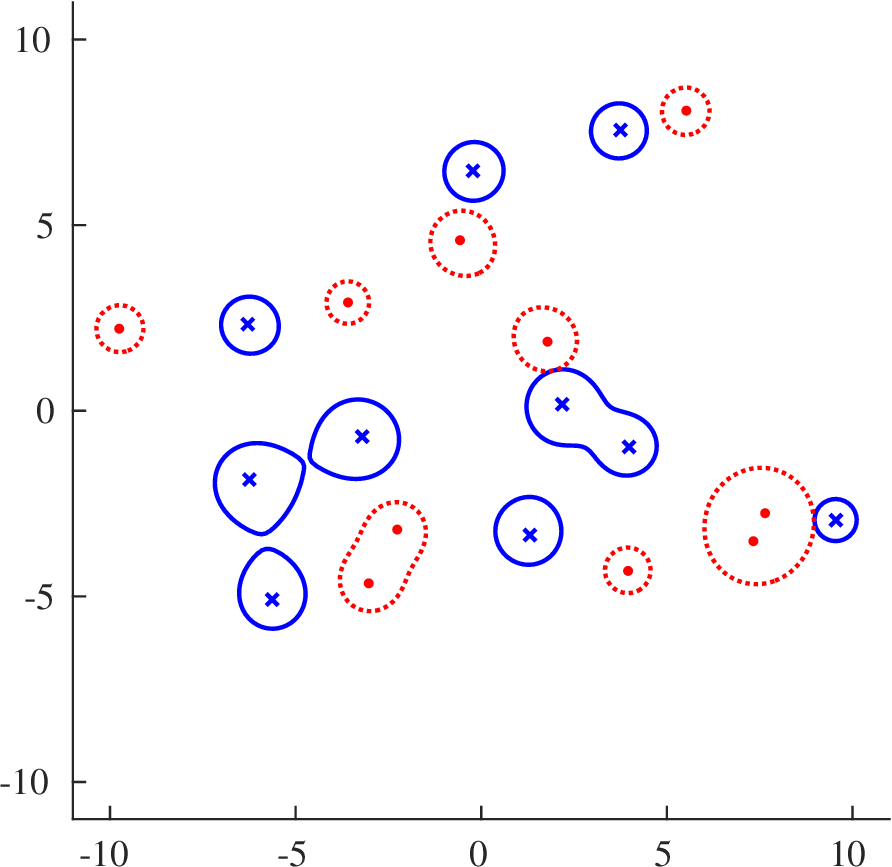} 
\label{fig:ps_sepld}
} 
\\
\subfloat[\seplvpscap]{
\includegraphics[scale=\scaledps,trim={0.0cm 0.0cm 0.0cm 0.0cm},clip]{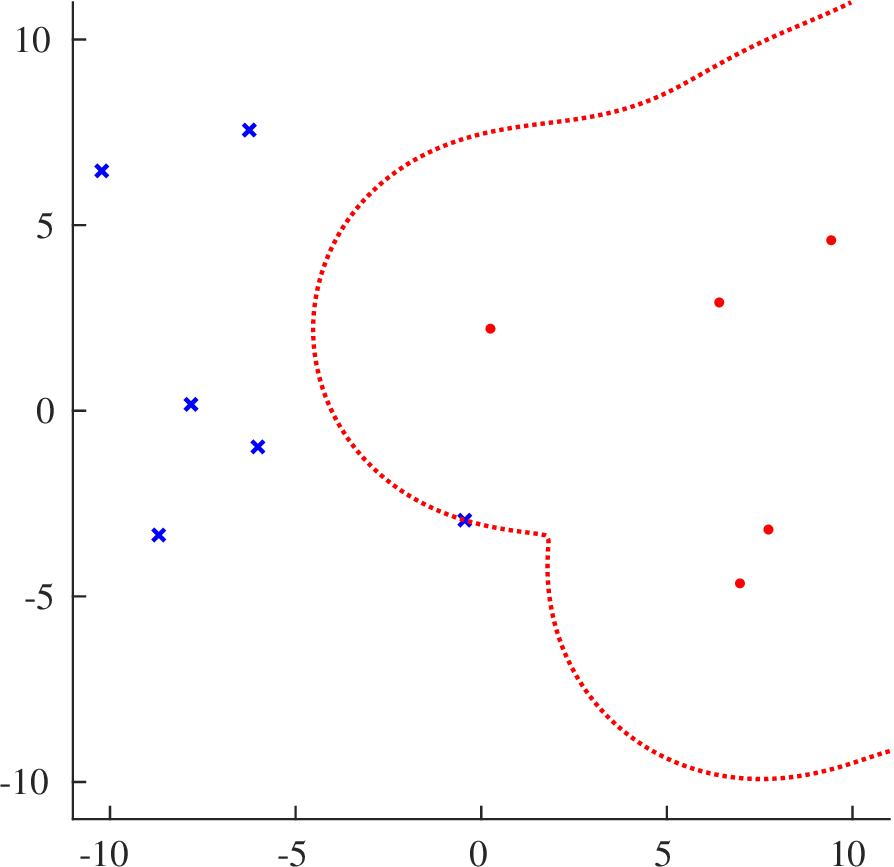} 
\includegraphics[scale=\scaledps,trim={0.cm 0.0cm 0.0cm 0.0cm},clip]{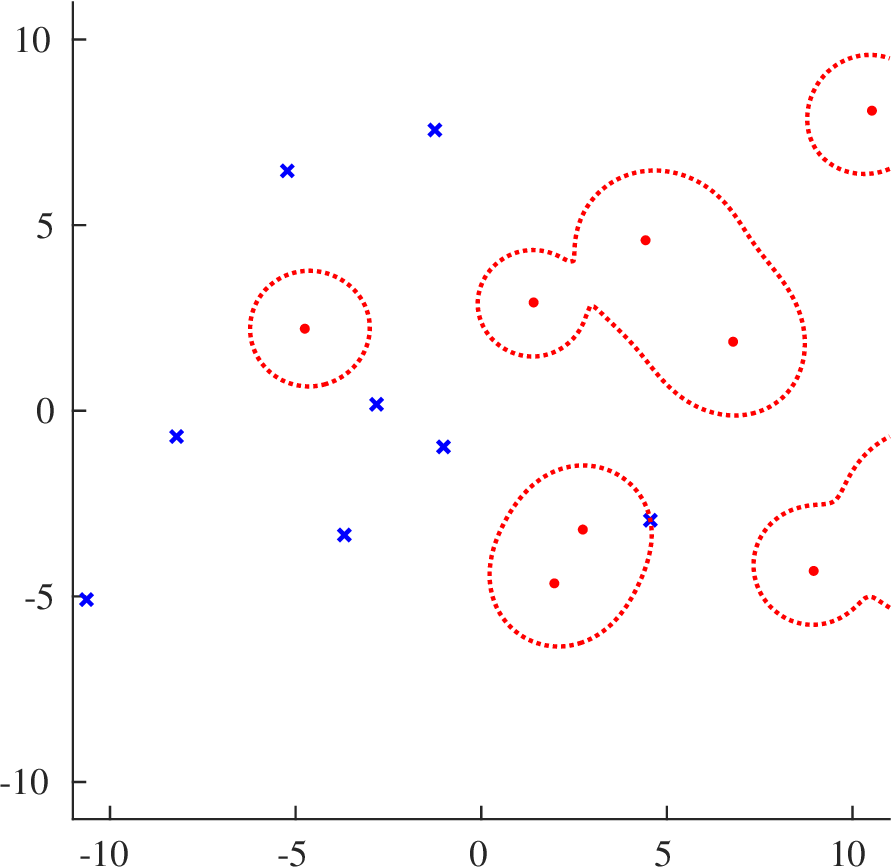} 
\includegraphics[scale=\scaledps,trim={0.0cm 0.0cm 0.0cm 0.0cm},clip]{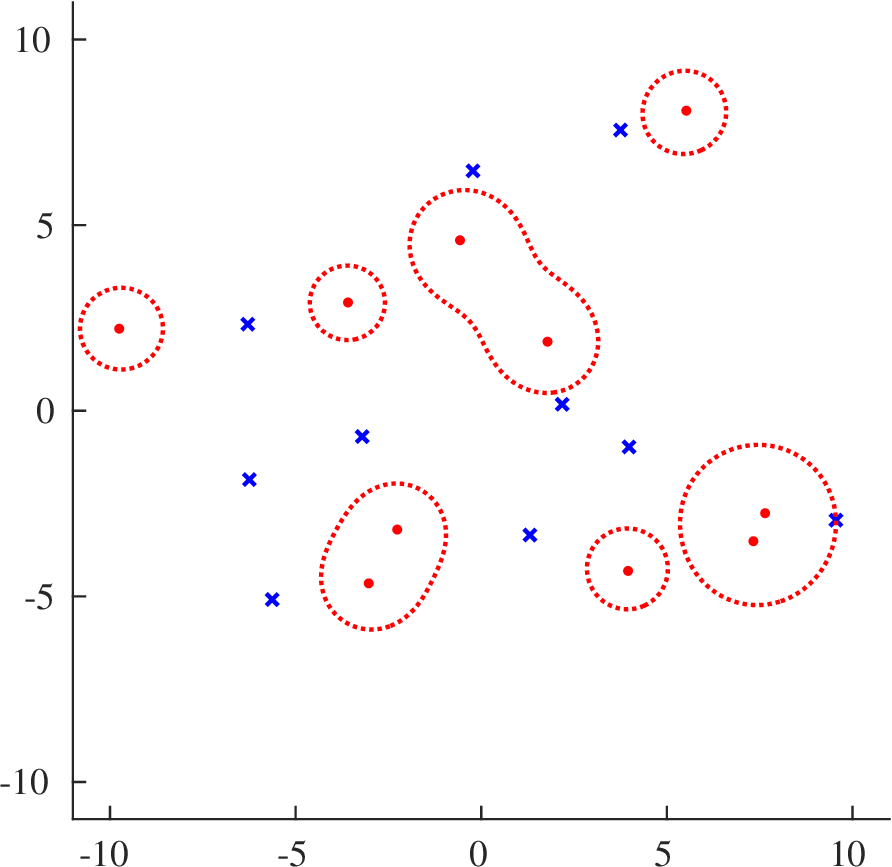} 
\label{fig:ps_seplv}
} 
\caption{
For the example described in~\cref{sec:explore},
pseudospectra of $A(s)$ and $B(s)$ (respectively solid and dotted contours)
corresponding to Demmel's and Varah's versions of sep-lambda are shown 
along with the eigenvalues of $A(s)$ and $B(s)$ (respectively x's and dots) for $s \in \{10,5,0\}$.
In the top right plot, $\ps(A(s))$ and $\ps(B(s))$ appear to touch at two places, but 
actually there is only one contact point (the one closer to the origin).
In the three lower plots, Varah's sep-lambda is attained with $\eps_1=0$.
}
\label{fig:shifted_ex_ps}
\end{figure}

\subsection{An exploratory example}
\label{sec:explore}
We first consider a simple example to explore the properties of our methods. 
We generated two different complex $10 \times 10$ matrices 
using \texttt{randn} and rescaled them so that the resulting matrices~$A$ and~$B$ 
both had spectral radii of 10.
Then, for $s \in \{10,5,0\}$, we considered Demmel's and Varah's versions 
of sep-lambda for $A(s) = A - sI$ and $B(s) = B + sI$.
When $s=0$, the spectra of~$A(s)$ and~$B(s)$ are ``centered" are the origin,
but when $s$ is increased, the centers of the two spectra,~$-s$ and~$s$, 
become more and more distant from each other; 
hence, on a macro level, increasing $s$ generally increases the value of $\seplamDAB{A(s)}{B(s)}$,
and this is always true once $s$ becomes sufficiently large.
Estimates of $\seplamDAB{A(s)}{B(s)}$ were computed using \cref{alg:seplamd},
while estimates of~$\seplamVAB{A(s)}{B(s)}$ were computed using both of our algorithms from~\cref{sec:varah};
for Varah's sep-lambda, the estimates for both our algorithms agreed exactly since 
they were obtained at an eigenvalue of $A(s)$.
For~\cref{alg:seplamd} and its extension to Varah's sep-lambda,
we used~\mbox{$10 + 10\imagunit$} as an initial point for optimization,
which was chosen so that some restarts would be observed.
In~\cref{fig:shifted_ex_ps}, we show the resulting pseudospectra of~$A(s)$ and~$B(s)$ 
at the perturbation levels given by~$\seplamDAB{A(s)}{B(s)}$ and~$\seplamVAB{A(s)}{B(s)}$.

We give performance statistics of \cref{alg:seplamd} on our exploratory example in \cref{table:overall}.
For~\mbox{$s=10$}, GRANSO found a global minimizer of $\fD$ from the initial point and so only a single certificate
computation was needed in this case, while two certificates were needed for the $s=5$ and $s=0$ instances.
On both of these, the first round of optimization only found a local minimizer, and so the first certificate instead returned
new points to restart optimization.  But as can be seen from~\cref{table:overall},
this happened with very little effort; only 15 evaluations of $d_\eps$ were needed to find new starting points.
\cref{fig:certs} shows that the corresponding final configurations of~$d_\eps$
are all nonnegative, as they should be when~$\eps = \seplamDAB{A(s)}{B(s)}$, per~\cref{cor:demmel_global}.
Overall, we see that additional effort was needed to approximate~$d_\eps$ as $s$ is decreased,
which is as we would expect because 
the behavior of~$d_\eps$ generally becomes more complicated in proportion to how much the two (pseudo)spectra 
``intermingle", which for our test examples, is roughly controlled by~$s$.
Recalling that the search point $z_0$ defining $d_\eps$ is near the origin for these problems, 
this effect can be clearly observed by looking at~\cref{fig:ps_sepld,fig:certs}
(and is also illustrated in~\cref{fig:comps_ex}, where each quadrant of the complex plane 
has a different amount of pseudospectral ``intermingling").
For~$s=10$, the eigenvalues of~$A(s)$ and $B(s)$ are separated from each other the most,
which in turn leads to the final~$d_\eps$  being rather straightforward; see \cref{fig:cert10}.
However, the separation between the eigenvalues of~$A(s)$ and~$B(s)$ is reduced via making~$s$ smaller,
and hence we see that $d_\eps$ becomes increasingly more complicated and with more discontinuities;
see \cref{fig:cert5} and \cref{fig:cert0}.

Performance data for our two algorithms for Varah's sep-lambda are given in~\cref{table:varah_algs},
where we see a similar effect with respect to changing shift $s$. 
However, the main takeaway here is that, as predicted, 
our algorithm from \cref{sec:varah_alg2} 
is indeed many times slower than our extension of~\cref{alg:seplamd} described in~\cref{sec:varah_alg1}.

\def\scaledfn{0.359}
\begin{figure}[!ht]
\centering
\subfloat[$s = 10$: $d_\eps(\theta)$ in linear scale (left) and in $\log_{10}$ scale (right)]{
\includegraphics[scale=\scaledfn,trim={0.0cm 0.0cm 0.0cm 0.0cm},clip]{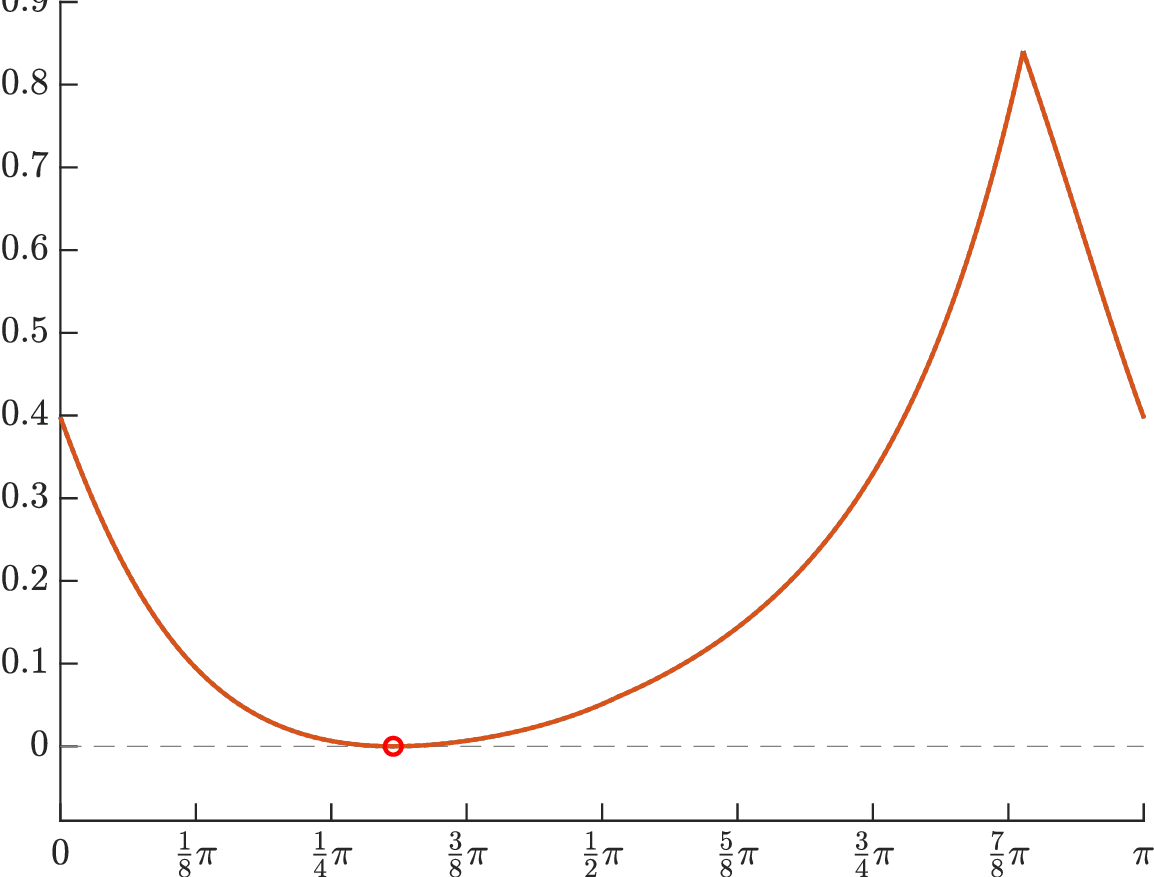} \ \ \
\includegraphics[scale=\scaledfn,trim={0.0cm 0.0cm 0.0cm 0.0cm},clip]{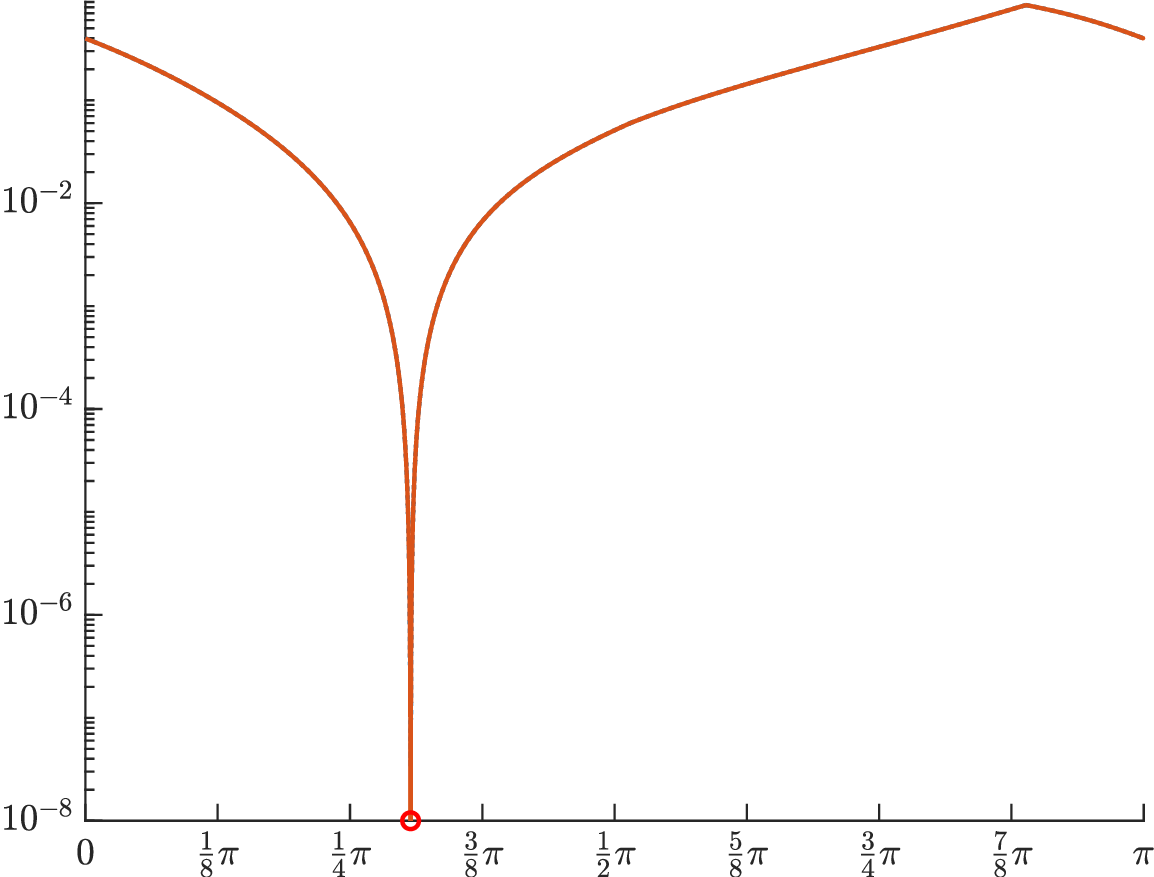} 
\label{fig:cert10}
} 
\\
\subfloat[$s = 5$: $d_\eps(\theta)$ in linear scale (left) and in $\log_{10}$ scale (right)]{
\includegraphics[scale=\scaledfn,trim={0.0cm 0.0cm 0.0cm 0.0cm},clip]{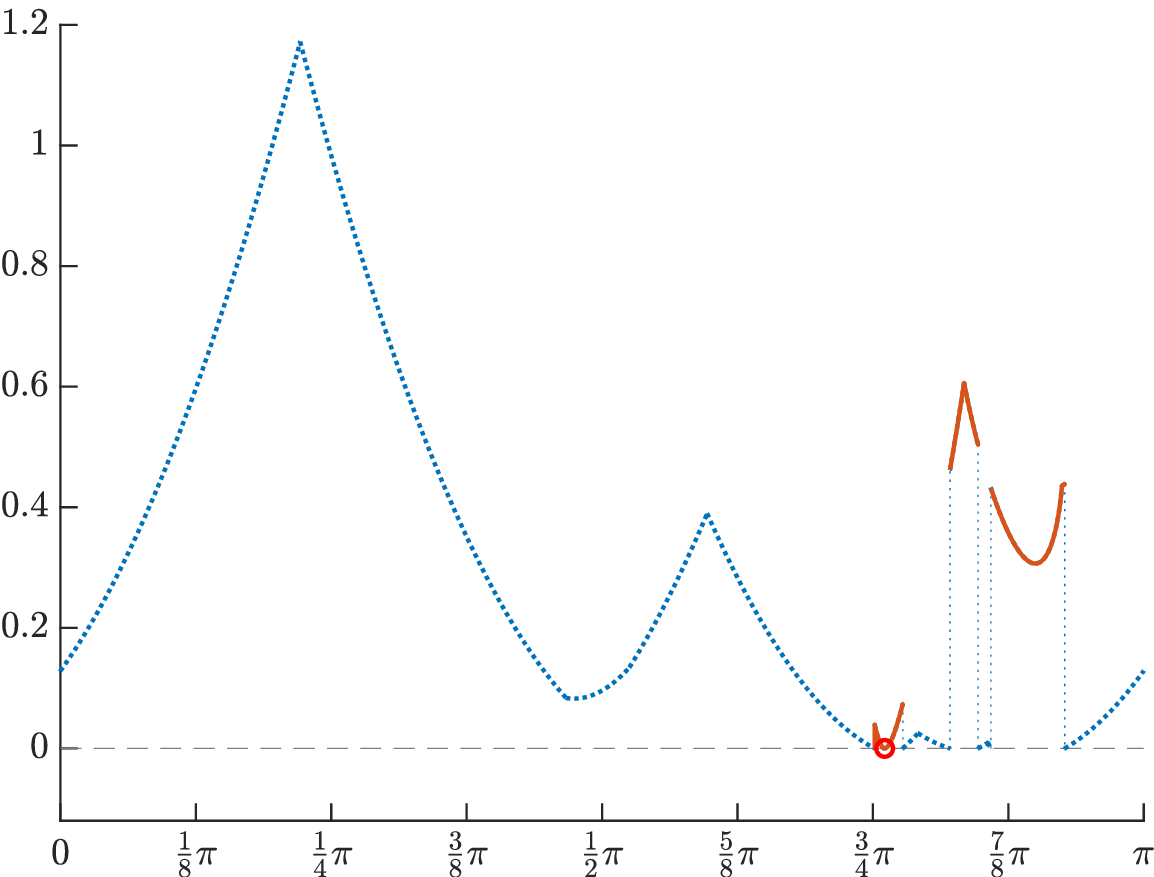} \ \ \
\includegraphics[scale=\scaledfn,trim={0.0cm 0.0cm 0.0cm 0.0cm},clip]{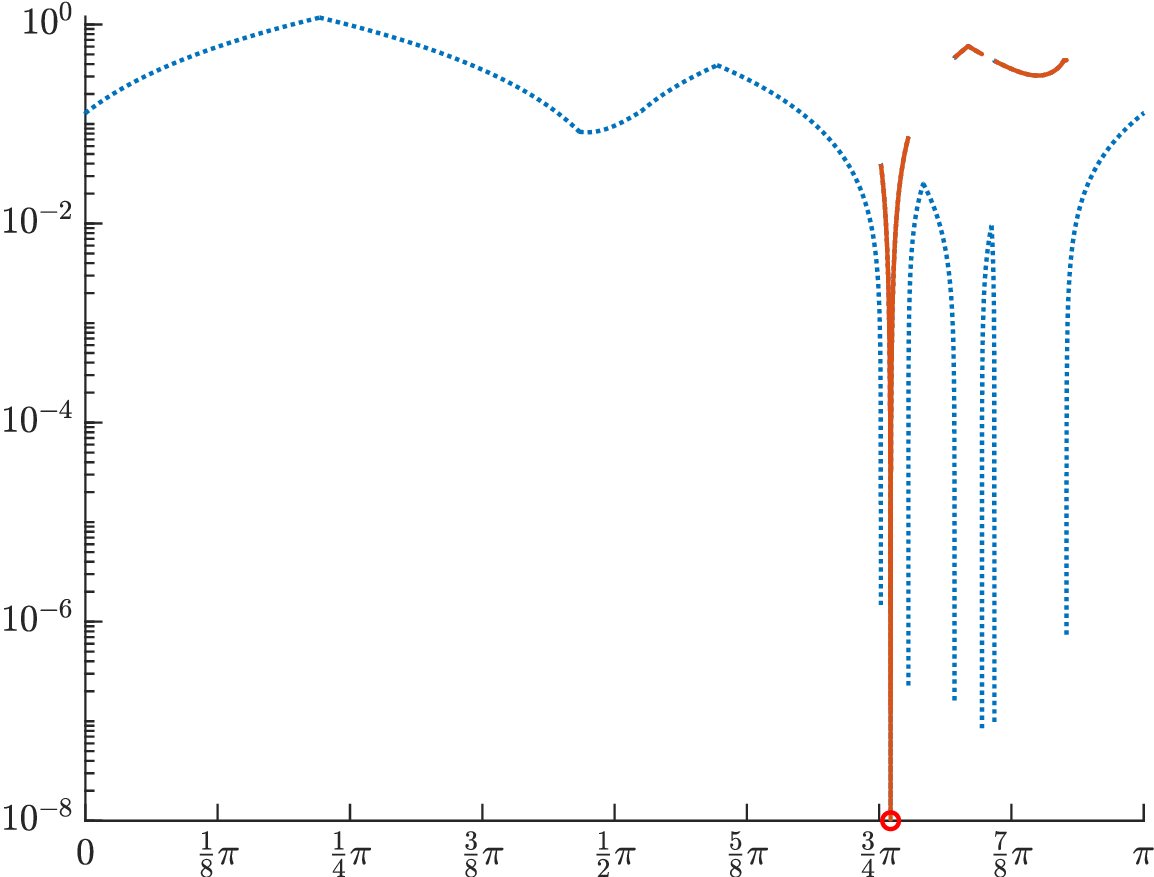}
\label{fig:cert5}
}
\\
\subfloat[$s = 0$: $d_\eps(\theta)$ in linear scale (left) and in $\log_{10}$ scale (right)]{
\includegraphics[scale=\scaledfn,trim={0.0cm 0.0cm 0.0cm 0.0cm},clip]{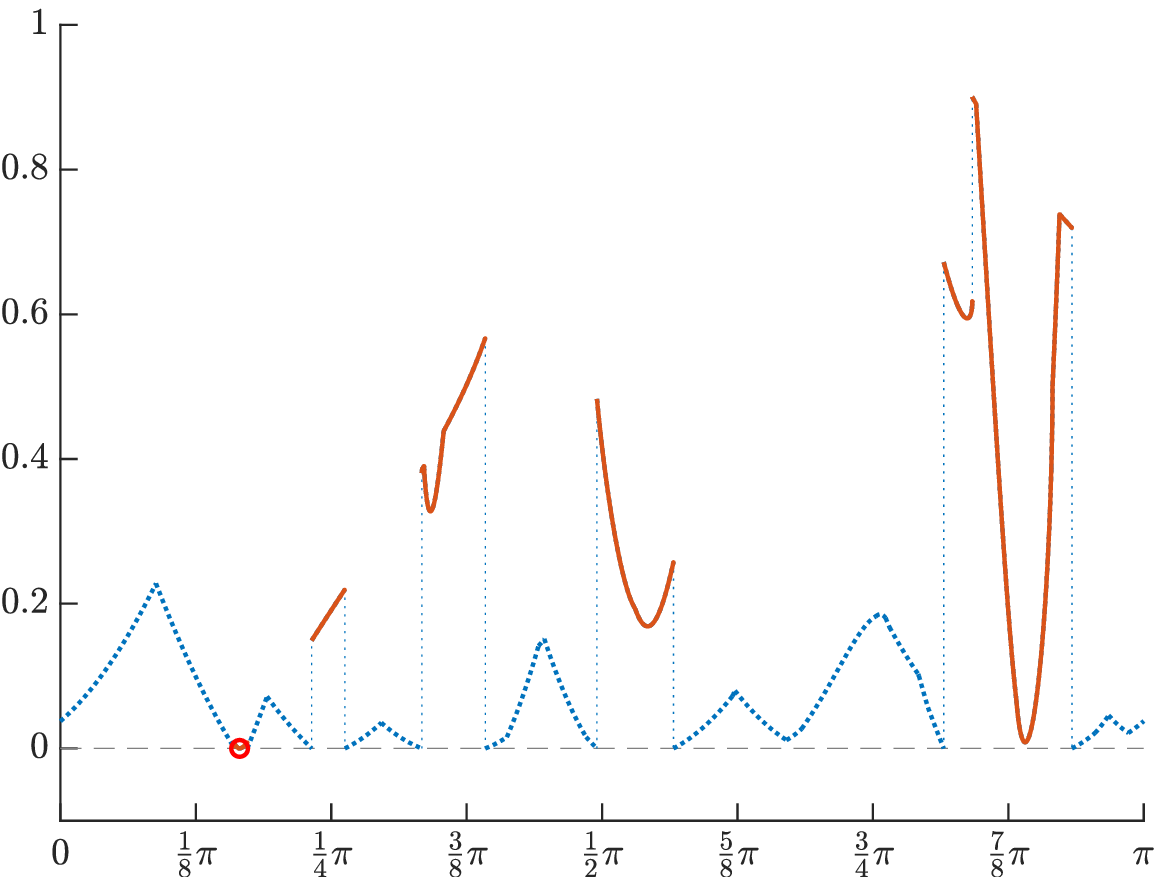} \ \ \
\includegraphics[scale=\scaledfn,trim={0.0cm 0.0cm 0.0cm 0.0cm},clip]{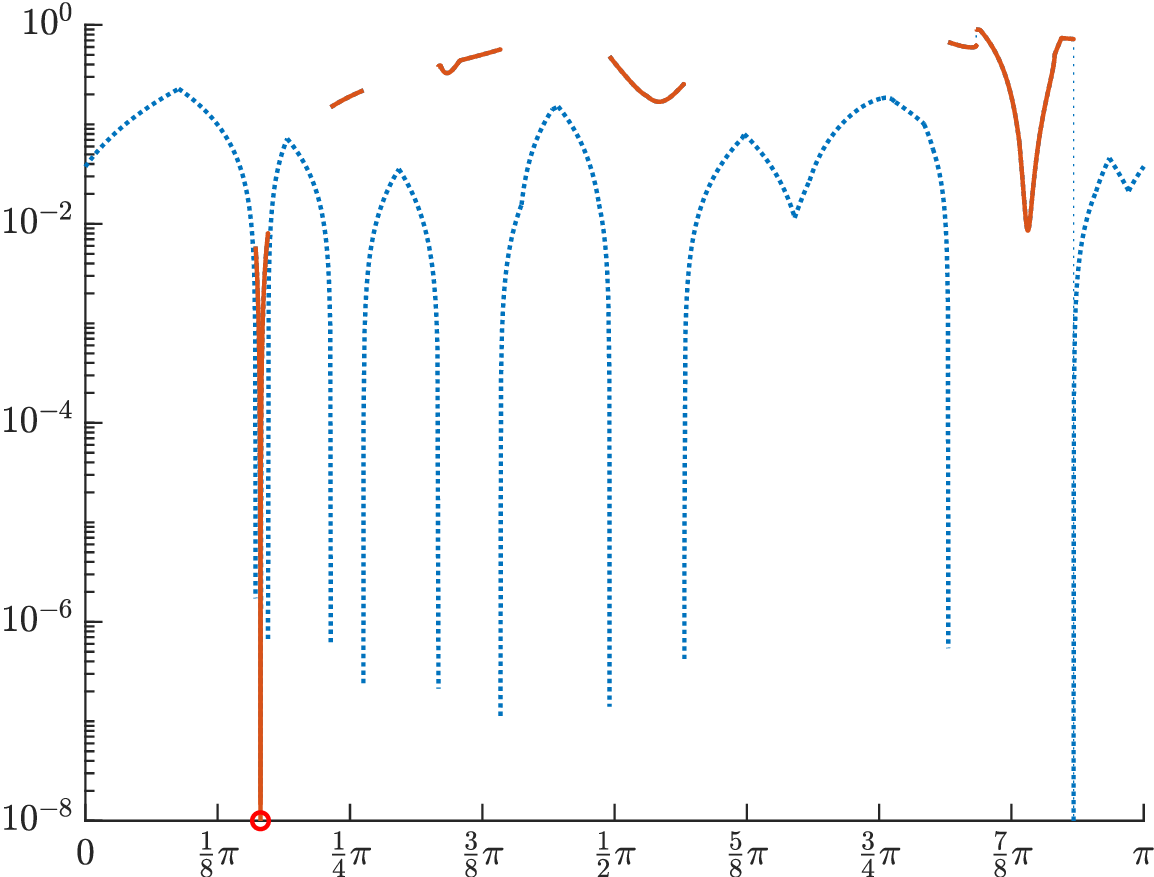} 
\label{fig:cert0}
}
\caption{
Each subfigure shows the final $d_\eps$ computed by \cref{alg:seplamd} 
for the example from \cref{sec:explore} for~$s \in \{10,5,0\}$.
The components of $d_\eps$ are plotted as follows: $\afn + \bfn$ (dotted)
and $d_\eps^{AB}$ (solid); $\olfn$ does not appear as it is never negative
when $\eps=\seplamD$.
The circle denotes the angle $\theta$ (with respect to $z_0$) 
associated with the best minimizer of $\fD$ obtained and corresponds to 
the single place where~$d_\eps(\theta) = 0$, which is more easily seen in the $\log_{10}$ plots on the right.
}
\label{fig:certs}
\end{figure}

\begin{table}[t]
\caption{
Performance data for \cref{alg:seplamd} for Demmel's version of sep-lambda
on the example given in~\cref{sec:explore} for $s \in \{10,5,0\}$:
``$\fD$~evals." is the total number of evaluations of $\fD$ during local optimization,
``Certs." is the total number of certificates attempted, 
``All" and ``Final" are the total number of evaluations of $d_\eps$ over all certificates and just the final one, respectively,
and the final column is the total running time in seconds of \cref{alg:seplamd}.
}
\centering
\setlength{\tabcolsep}{6.5pt} 
\begin{tabular}{ c | c | c | rr | D{.}{.}{2} } 
\toprule
\multicolumn{1}{c}{} & \multicolumn{1}{c}{}  & \multicolumn{1}{c}{}  & 
 \multicolumn{2}{c}{$d_\eps$ evals.} & \multicolumn{1}{c}{Time (sec.)} \\
\cmidrule(lr){4-5} 
\cmidrule(lr){6-6} 
	\multicolumn{1}{c}{$s$} & 
	\multicolumn{1}{c}{$\fD$ evals.} &
	\multicolumn{1}{c}{Certs.} & 
	\multicolumn{1}{c}{All} & 
	\multicolumn{1}{c}{Final} &
	\multicolumn{1}{c}{Alg.~\ref{alg:seplamd}} \\
\midrule
  10 &    156 &   1 &     2154 &     2154 &     4 \\ 
   5 &    150 &   2 &     8587 &     8572 &     5 \\ 
   0 &    205 &   2 &    22838 &    22823 &    13 \\ 
\bottomrule
\end{tabular}
\label{table:overall}
\end{table}

\begin{table}[t]
\caption{
Performance data for our two algorithms described in \cref{sec:varah}
for Varah's version of sep-lambda
on the example given in~\cref{sec:explore} for $s \in \{10,5,0\}$.
For the extension of \cref{alg:seplamd} described in \cref{sec:varah_alg1},
``$\fV$~evals." is the total number of evaluations of $\fV$ during all optimization runs,
``Certs." is the total number of certificates attempted,
and ``$d_{\eps_1,\eps_2}$~evals." is the total number of evaluations of $d_{\eps_1,\eps_2}$ over all certificates.
For our algorithm described in \cref{sec:varah_alg2}, 
``$\fV$~evals." is the total number of evaluations of~$\fV$,
``$v$~evals." is the total number of evaluated of the function $v$ defined in~\eqref{eq:v_theta}.
Finally, we report the total running time in seconds for each method respectively 
under the ``\cref{sec:varah_alg1}" and ``\cref{sec:varah_alg1}" columns.
}
\centering
\setlength{\tabcolsep}{6.5pt} 
\begin{tabular}{ c | c | c | D{.}{.}{2} | r | c | cc } 
\toprule
	\multicolumn{1}{c}{} & 
	\multicolumn{3}{c}{Alg. from \cref{sec:varah_alg1}} & 
	\multicolumn{2}{c}{Alg. from \cref{sec:varah_alg2}} &
	\multicolumn{2}{c}{Time (sec.)} \\
\cmidrule(lr){2-4} 
\cmidrule(lr){5-6} 
\cmidrule(lr){7-8} 
	\multicolumn{1}{c}{$s$} & 
	\multicolumn{1}{c}{$\fV$ evals.} &
	\multicolumn{1}{c}{Certs.} & 
	\multicolumn{1}{c}{$d_{\eps_1,\eps_2}$ evals.} & 
	\multicolumn{1}{c}{$\fV$ evals.} &
	\multicolumn{1}{c}{$v$ evals.} & 
	\multicolumn{1}{c}{\cref{sec:varah_alg1}} &
	\multicolumn{1}{c}{\cref{sec:varah_alg2}} \\
\midrule
  10 &    188 &   1 &   4494 &    2512417 &   2044 &    2 &  139 \\ 
   5 &    142 &   1 &   4583 &   10161448 &   5900 &    2 &  548 \\ 
   0 &    150 &   1 &   6859 &   14736950 &   6502 &    2 &  777 \\ 
\bottomrule
\end{tabular}
\label{table:varah_algs}
\end{table}

\subsection{Comparing \cref{alg:seplamd} to the method of Gu and Overton}
We now do a comparison of~\cref{alg:seplamd} 
against the \texttt{seplambda} routine\footnote{Available at \url{https://cs.nyu.edu/faculty/overton/software/seplambda/}.},
which is Overton's \matlab\ implementation of his $\seplamD$ algorithm with Gu \cite{GuO06a}.
To do this, we generated two more examples in the manner as described in \cref{sec:explore} but now for $m=n=20$
and $m=n=40$.  For each, including our earlier \mbox{$m=n=10$} example, we computed $\seplamDAB{A(s)}{B(s)}$
for~$s=0$ and~$s=m=n$ using both our new method and \texttt{seplambda}.
In order to obtain  $\seplamDAB{A(s)}{B(s)}$ to high precision, 
 we set the respective tolerances for both methods to~$10^{-14}$.
For this comparison, we always initialized the first phase of optimization for our method from the origin.

\begin{table}[t]
\caption{Comparing \cref{alg:seplamd} and the method of Gu and Overton for Demmel's version of sep-lambda.  
The columns are the same as described in \cref{table:overall}
except that we now additionally give
the problem size under ``$m=n$", 
the total running times in seconds of both methods, respectively ``Alg.~\cref{alg:seplamd}" and ``GO", 
and the relative difference between the estimates computed by both methods (Rel. Diff.),
with positive indicating our method returned a better (lower) equal estimate for $\seplamDAB{A(s)}{B(s)}$.
}
\centering
\setlength{\tabcolsep}{6.5pt} 
\begin{tabular}{ c | c | c | c | rr | rr | c } 
\toprule
\multicolumn{1}{c}{} &\multicolumn{1}{c}{} & \multicolumn{1}{c}{}  & \multicolumn{1}{c}{}  &  
\multicolumn{2}{c}{$d_\eps$ evals.} & \multicolumn{2}{c}{Time (sec.)} & \multicolumn{1}{c}{} \\
\cmidrule(lr){5-6} 
\cmidrule(lr){7-8} 
	\multicolumn{1}{c}{$m=n$} & 
	\multicolumn{1}{c}{$s$} & 
	\multicolumn{1}{c}{$\fD$ evals.} &
	\multicolumn{1}{c}{Certs.} & 
	\multicolumn{1}{c}{All} & 
	\multicolumn{1}{c}{Final} &
	\multicolumn{1}{c}{Alg.~\ref{alg:seplamd}} &
	\multicolumn{1}{c}{GO} &
	\multicolumn{1}{c}{Rel. Diff.} \\
\midrule
  10 &   10 &    65 &   1 &    2154 &    2154 &       3 &      16 & $3.5 \times 10^{-14}$ \\ 
  10 &    0 &    75 &   1 &   23287 &   23287 &      14 &      15 & $0$ \\ 
  20 &   20 &    97 &   1 &    4746 &    4746 &      15 &    1481 & $1.1 \times 10^{-13}$ \\ 
  20 &    0 &   346 &   3 &   31786 &   31756 &      78 &    1408 & $1.3 \times 10^{-13}$ \\ 
  40 &   40 &   128 &   2 &    5973 &    5910 &      95 &  237425 & $3.3 \times 10^{-14}$ \\ 
  40 &    0 &   295 &   3 &   29261 &   29231 &     407 &  230251 & $2.0 \times 10^{-12}$ \\ 
\bottomrule
\end{tabular}
\label{table:comp}
\end{table}

A performance overview is reported in \cref{table:comp}.  In terms of accuracy,
the estimates computed by our method for~$\seplamDAB{A(s)}{B(s)}$ have high agreement with 
those computed by \texttt{seplambda}, though our method did return slightly better (lower) values
for all the problems.  On the nonshifted~($s=0$) examples, 
our new method was 1.1~times faster than \texttt{seplambda} for $m=n=10$, 
18.1~times faster for $m=n=20$, and 566.0~times faster for $m=n=40$.  
Clearly, as the problems get larger, our method will be even faster relative to \texttt{seplambda}.
For the shifted examples ($s=m=n$), the performance gaps are even wider:
our new method was 6.2~times faster than \texttt{seplambda} for $m=n=10$, 
98.6~times faster for $m=n=20$, and 2495.7~times faster for~$m=n=40$.  
The ``$d_\eps$ evals." data for $s=0$ and $s=m=n$ in \cref{table:comp}
for these problems also 
indicate that $d_\eps$ is generally less complex the more the eigenvalues of~$A$ and~$B$ are separated.
Meanwhile, the running times of \texttt{seplambda} were relatively unchanged by the value of~$s$, 
as shifting the eigenvalues of $A$ and $B$ has no direct effect on its computations.
In \cref{table:comp}, we can again infer that restarts in our method, when needed, 
happened with relatively few evaluations of~$d_\eps$.
Per \cite[Section~5.2]{Mit21}, the main work done in
interpolation-based globality certificates is ``embarrassingly parallel",
and consequently, our method can further be accelerated by about an order of magnitude using parallel processing,
and substantially more if minor tweaks are made to Chebfun to make it 
more amenable to parallelism.

\begin{remark}
Recalling the end of~\cref{sec:background}
on possibly replacing the bisection phase of Gu and Overton's with
optimization-with-restarts, we now empirically validate
our claim that the benefit of such a modification is indeed quite limited and diminishes as 
the problem dimensions increase. 
Besides recording the total time to run \texttt{seplambda} on each problem for~\cref{table:comp},
we also recorded the time its initialization procedure required.
Then, an upper bound for the best possible speedup is simply 
the total time divided by the initialization time,
where we idealistically assume that opimization-with-restarts has zero cost.
For  $m=n$ respectively equal to $10$, $20$, and $40$, the computed ratios 
were approximately 3.5, 2.2, and~1.6.
Obviously, even these idealized speedups are nowhere near sufficient to 
overcome the very large performance gaps shown in~\cref{table:comp}
for $m=n=20$, let alone $m=n=40$, although such a modified version of \texttt{seplambda} 
would be close in performance to our method on the $m=n=10$, $s=10$ problem
and likely pull ahead for the $m=n=10$,~$s=0$ problem.
However, if we enabled parallel processing for \cref{alg:seplamd},
then it would again be fastest on this problem too and probably by a large margin.
Finally, note that if \texttt{seplambda} were further modified by
also adapting the divide-and-conquer technique of~\cite{GuMOetal06},
it still would be significantly slower than~\cref{alg:seplamd}, except for maybe the tiniest of problems.
In the context of computing the distance to uncontrollability,
we compared our interpolation-based globality certificates methodology with the method 
of~\cite{GuMOetal06}, which uses both optimization-with-restarts and divide-and-conquer and also
does not have any expensive initialization procedure, and our approach
was roughly $5$ to $43$ times faster depending on the dimension; see~\cite[Section~5.1]{Mit21}.
\end{remark}

\subsection{Scaling performance of \cref{alg:seplamd}}
Finally, we examine the scaling performance of~\cref{alg:seplamd} on some larger problems,
which we constructed in the same fashion as before except that here we generated 
complex matrices~$A$ and $B$~via \texttt{sprandn} with a density of $0.1$;
this change was done solely to be able to store the matrices explicitly 
while keeping the file sizes small for up to $m=n=800$.
For these problem sizes, it was not feasible to attempt running Gu and Overton's method,
so in~\cref{table:scaling}, we only give performance data for~\cref{alg:seplamd}.
The accuracy of each estimate $\eps$ for $\seplamDAB{A(s)}{B(s)}$ computed by~\cref{alg:seplamd}
was verified by creating a sufficiently high resolution plot of $\ps(A(s))$ and~$\ps(B(s))$ 
and inspecting it to see whether or not the interiors of the two pseudospectra overlap.
This visual check suffices to confirm the high accuracy of our new method because, per~\cref{sec:find_local_mins}, 
local minimizers discovered on every iteration of~\cref{alg:seplamd}
will be computed to high accuracy, and the fact 
that~$\eps > \seplamDAB{A(s)}{B(s)}$ if and only 
if~\mbox{$\interior \ps(A(s)) \cap \interior \ps(B(s)) \neq \varnothing$};
hence, to assess the accuracy of a computed estimate~$\eps$, we need only confirm whether or not~\cref{alg:seplamd}
converged to a global minimizer of $\fD$ or only a local one, which is done by looking for the
absence or presence, respectively, of pseudospectral overlap.
For the pair of smallest problems~($m=n=100$), \cref{alg:seplamd} respectively
took about $11$ and $50$ minutes, while on the other extreme,
 \cref{alg:seplamd} needed about~$6$ and~$37$ hours, respectively,
 for the two $m=n=800$ problem instances.
Again, using parallel processing can reduce these 
running times dramatically.  
 Interestingly, for the intermediate sizes of $m=n=200$ and $m=n=400$,
 we actually see that \cref{alg:seplamd} was slightly more expensive on the instances
 with nonzero $s$, which suggests that the spectra of $A(s)$ and $B(s)$ 
for these particular examples would need to be shifted even further apart 
in order for the complexity of $d_\eps$ to decrease.
Over all the problems tested, we see that~\cref{alg:seplamd}
required at most four restarts before converging, but once again, the costs of these restarts 
was generally negligible, with the one exception being the $m=n=800$, $s=0$
problem, where we can infer that the total cost of the four restarts was approximately 
$10\%$ of the overall running time.

\begin{table}[t]
\caption{The columns are the same as described in \cref{table:comp}
except that here we only give running times of \cref{alg:seplamd} and its computed estimates 
of $\seplamDAB{A(s)}{B(s)}$.
The running time and accuracy comparisons with Gu and Overton's method 
are not provided since it would have taken far too long to run their method on these larger problems.
}
\centering
\setlength{\tabcolsep}{6.5pt} 
\begin{tabular}{ c | c | c | c | rr | r | l } 
\toprule
\multicolumn{1}{c}{} &\multicolumn{1}{c}{} & \multicolumn{1}{c}{}  & \multicolumn{1}{c}{}  &  
\multicolumn{2}{c}{$d_\eps$ evals.} & \multicolumn{1}{c}{Time (sec.)} & \multicolumn{1}{c}{} \\
\cmidrule(lr){5-6} 
\cmidrule(lr){7-7} 
	\multicolumn{1}{c}{$m=n$} & 
	\multicolumn{1}{c}{$s$} & 
	\multicolumn{1}{c}{$\fD$ evals.} &
	\multicolumn{1}{c}{Certs.} & 
	\multicolumn{1}{c}{All} & 
	\multicolumn{1}{c}{Final} &
	\multicolumn{1}{c}{Alg.~\ref{alg:seplamd}} &
	\multicolumn{1}{c}{$\seplamDAB{A(s)}{B(s)}$} \\
\midrule
 100 &  100 &    89 &   1 &    5425 &    5425 &     704 & $6.1677176880084 \times 10^{0}$ \\ 
 100 &    0 &   334 &   4 &   23689 &   23451 &    3045 & $2.5004731832266 \times 10^{-2}$ \\ 
 200 &  200 &   210 &   2 &   23170 &   23155 &   11563 & $6.3206868631252 \times 10^{0}$ \\ 
 200 &    0 &   319 &   2 &   20846 &   20831 &   10187 & $4.2654521922541 \times 10^{-2}$ \\ 
 400 &  400 &   123 &   2 &   22456 &   22425 &   40966 & $6.0394981396743 \times 10^{0}$ \\ 
 400 &    0 &   392 &   5 &   18237 &   18113 &   33931 & $4.0258158186612 \times 10^{-2}$ \\ 
 800 &  800 &   153 &   2 &    3113 &    3098 &   22112 & $1.0584889222355 \times 10^{1}$ \\ 
 800 &    0 &   383 &   5 &   21962 &   19436 &  131887 & $9.9483548512835 \times 10^{-3}$ \\ 
\bottomrule
\end{tabular}
\label{table:scaling}
\end{table}

\section{Concluding remarks}
\label{sec:conclusion}
In this paper, we have introduced a new method to compute Demmel's version of sep-lambda
that is much faster than the only previous known algorithm (due to Gu and Overton).
Under our assumption that approximation of $d_\eps$ by interpolation 
is reliable, our method computes $\seplamD$ to arbitrary accuracy and generally 
behaves like a method with cubic work complexity, albeit one with a high constant factor.
Nevertheless, our new approach is so much faster that it is now possible to calculate $\seplamD$ for moderately sized problems,
e.g., for~$m,n$ in the thousands,
which were simply intractable when using Gu and Overton's algorithm.
We have also extended our algorithm to tackle Varah's version of sep-lambda. 
Although in this case global optimality cannot be guaranteed, 
the extension does rapidly compute locally optimal approximations which satisfy the necessary condition for global optimality.
Furthermore, we have a proposed a second method to 
actually compute~$\seplamV$, although this algorithm is significantly more expensive.

\section*{Acknowledgements}
The author is grateful for the referees' helpful suggestions,
and particularly for the first referee's observation that the efficient pseudospectral plotting technique
of~\cite{Lui97} could be used in~\cref{alg:seplamd} to bring down its worst-case 
asymptotic work complexity to cubic in $n$ as opposed to quartic; see~\cref{sec:d_eps_eval} and \cref{sec:work}.

\small
\bibliographystyle{alpha}
\bibliography{csc,special,software}

\end{document}